%% file: ICML2023.tex
\documentclass{article}

\usepackage{microtype}
\usepackage{graphicx}
\usepackage{subfigure}
\usepackage{booktabs} 

\usepackage{hyperref}

\usepackage[accepted]{icml2023_off}

\usepackage{amsmath}
\usepackage{amssymb}
\usepackage{mathtools}
\usepackage{amsthm}

\usepackage[capitalize,noabbrev]{cleveref}

\usepackage{url}            
\usepackage{booktabs}       
\usepackage{amsfonts}       
\usepackage{nicefrac}       
\usepackage{microtype}      

\theoremstyle{plain}

\usepackage{color}
\usepackage{amssymb}
\usepackage{amsthm}
\usepackage{graphicx}
\usepackage{dsfont}

\usepackage{subfigure}

\input{macros_old}

\makeatletter

\icmltitlerunning{Stochastic Gradient Descent under Markovian Sampling Schemes}

\begin{document}

\twocolumn[
\icmltitle{Stochastic Gradient Descent under Markov-Chain Sampling Schemes}



\icmlsetsymbol{equal}{*}

\begin{icmlauthorlist}
\icmlauthor{Mathieu Even}{in}
\end{icmlauthorlist}

\icmlaffiliation{in}{Inria - ENS Paris}

\icmlcorrespondingauthor{Mathieu Even}{mathieu.even@inria.fr}

\icmlkeywords{Optimization, random walks, decentralized, distributed, convex, stochastic, markov chain}

\vskip 0.3in
]



%
\printAffiliationsAndNotice{}  

\begin{abstract} 
We study a variation of vanilla stochastic gradient descent where the optimizer only has access to a ‘‘Markovian sampling scheme''. These schemes encompass applications that range from decentralized optimization with a random walker (\emph{token algorithms}), to RL and online system identification problems. We focus on obtaining rates of convergence under the least restrictive assumptions possible on the underlying Markov chain and on the functions optimized. We first unveil the theoretical lower bound for methods that sample stochastic gradients along the path of a Markov chain, making appear a dependency in the hitting time of the underlying Markov chain. We then study \emph{Markov chain SGD} (MC-SGD) under much milder regularity assumptions than prior works. We finally introduce MC-SAG, an alternative to MC-SGD with variance reduction, that only depends on the hitting time of the Markov chain, therefore obtaining a communication-efficient token algorithm.
\end{abstract}

\section{Introduction}

In this paper, we consider a stochastic optimization problem that takes root in decentralized optimization, estimation problems, and Reinforcement Learning.
Consider a function $f$ defined as:
\begin{equation}\label{eq:f_gen}
    f(\xx)=\E_{v\sim\pi}\left[ f_v(\xx)\right]\,,\quad \xx\in\R^d\,,
\end{equation}
where $\pi$ is a probability distribution over a set $\cV$, and $f_v$ are smooth functions on $\R^d$ for all $v$ in~$\cV$.
Classicaly, this represents the loss of a model parameterized by $\xx$ on data parameterized by $v$.
If \emph{i.i.d.}~samples $(v_t)_{t\geq0}$ of law $\pi$ and their corresponding gradient estimates $(\nabla f_{v_t})$ were accessible, one could directly apply SGD-like algorithms, that have proved to be efficient in large scale machine learning problems \citep{Bottou2018}.
We however consider in this paper a different setting: we assume the existence of a Markov chain $(v_t)$ of state space $\cV$ and stationary distribution $\pi$.
The optimizer may then use biased stochastic gradients along the path of this Markov chain to perform incremental updates. 
She may for instance use the \emph{Markov chain SGD} (MC-SGD) algorithm, defined through the following recursion:
\begin{equation}\label{eq:mcsgd}
    \xx_{t+1}=\xx_t-\gamma \nabla f_{v_t}(\xx_t)\,.
\end{equation}
Being ‘‘ergodically unbiased'', such iterates should behave closely to those of vanilla SGD.
The analysis is however notoriously difficult, since in~\eqref{eq:mcsgd}, variable $\xx_t$ and the current state of the Markov chain $v_t$ are not independent, so that $\esp{\nabla f_{v_t}(\xx_t)|\xx_t}$ can be arbitrarily far from $\nabla f(\xx_t)$.
This paper focuses on analyzing algorithms that incrementally sample stochastic gradients alongside the Markov chain $(v_t)$, motivated by the following applications.

\subsection{Token algorithms}
Traditional machine learning optimization algorithms require data centralization, raising scalability and privavy issues, hence the alternative of Federated Learning, where users' data is held on device, and the training is orchestrated at a server level.
Decentralized optimization goes further, by removing the dependency over a central entity, leading to increased scalability, privacy and robustness to node failures, broadening the range of applications to IoT (Internet of Things) networks.
In decentralized optimization, users (or agents) are represented as nodes of a connected graph $G=(\cV,\cE)$ over a finite set of users $\cV$ (of cardinality $n$). 
The problem considered is then the minimization of
\begin{equation}\label{eq:f}
    f(\xx)=\frac{1}{n}\sum_{v\in\cV}f_v(\xx)\,,\quad \xx\in\R^d\,,
\end{equation} 
where each $f_v$ is locally held by user $v\in\cV$, using only communications between neighboring agents in the graph.
There are several known decentralized algorithmic approaches to minimize $f$ under these constrains. 
The prominent one consists in alternating between communications using gossip matrices \citep{boyd2006gossip,dimakis2010synchgossip} and local gradient computations, until a consensus is reached.
These gossip approaches suffer from a high synchronization cost (nodes in the graph are required to perform simultaneous operations, or to be aware of operations at the other end of the communication graph) that can be prohibitive if we aim at removing the dependency on a centralized orchestrator.
Further, a high number of communications are required to reach consensus, whether all nodes in the graph (as in synchronous gossip) or only two neighboring ones (as in randomized gossip) communicate at each iteration.
To alleviate these communication burdens, based on the original works of \citet{lopes2007incremental,johansson2007simple,johansson2010randomizedincremental}, we study algorithms based on Markov chain SGD: a variable $\xx$ performs a random walk on graph $G$, and is incrementally updated at each step of the random walk, using the local function available at its location.
This approach thus boils down to the one presented above with the function defined in \eqref{eq:f_gen}, where $\cV$ is the (finite) set of agents, $\pi$ is the uniform distribution over $\cV$, and $(v_t)$ is the Markov chain consisting of the consecutive states of the random walk performed on graph $G$.
The random walk guarantees that every communications are spent on updating the global model, as opposed to gossip-based algorithms, where communications are used to reach a running consensus while locally performing gradient steps.

These algorithms are referred to as \emph{token algorithms}: a token (that represents the model estimate) randomly walks the graph and performs updates during its walk.
There are two directions to design and analyze token algorithms. 
\citet{johansson2007simple} designed and analyzed its algorithm using, based on SGD with subdifferentials and a Markov chain sampling (consisting of the random walk).   
Following works \citep{duchi2012ergodic,tao2018mcsgd} tried to improve convergence guarantees of such stochastic gradient algorithms with Markov chain sampling, under various scenarii (mirror SGD \emph{e.g.}).
However, all these analyses rely on overly strong assumption: bounded gradients and/or bounded domains are assumed, and the rates obtained are of the form $\tau_\mix/T+\sqrt{\tau_\mix/T}$ for a number $T$ of steps, where $\tau_\mix$ is the mixing time of the underlying Markov chain.
More recently, \citet{pmlr-v162-dorfman22a} obtained similar rates under similar assumptions (bounded losses and gradients), but without requiring any prior knowledge of $\tau_\mix$, using adaptive stepsizes.

A more recent approach consists in deriving token algorithms from Lagrangian duality and from variants of coordinate gradient methods or ADMM algorithms with Markov chain sampling. \citet{walkman} introduce the \emph{Walkman} algorithm, whose analysis works on any graph, and obtain rates of $\frac{\tau_\mix^2 n}{T}$ to reach approximate-stationary points, while \citet{hendrikx2022tokens} introduced a more general framework, but whose analysis only works on the complete graph (and is thus equivalent to an \emph{i.i.d.} sampling). 
Yet, \citet{hendrikx2022tokens} extend their analysis to arbitrary graph, by performing gradient updates every $\tau_\mix$ steps of the random walk, obtaining a a dependency on $\tau_\mix n$, making their algorithm state of the art for these problems. 
Altenatively, \citet{wang2022stability} studies the algorithm stability of MC-SGD in order to derive generalization upper-bounds for this algorithm, and \citet{pmlr-v162-sun22b} provides and studies adaptive token algorithms.
Recently, and concurrently to this work, \citet{doan2023mcsgd} also studies MC-SGD without smoothness; however, their dependency on the mixing time of the random walk (in their Theorem 1) scales as $\exp(c \tau_{\mix})$: this is prohibitive as soon as the mixing time becomes larger than $\cO(1)$.

In summary, current token algorithms and their analyses either rely on strong noise and regularity assumptions (\emph{e.g.} bounded gradients), or suffer from an overly strong dependency on Markov chain-related quantities (as in \cite{walkman,hendrikx2022tokens}).

The token algorithms we consider are to be put in contrast with \emph{consensus-based decentralized} algorithms, or \emph{gossip} algorithms (with fixed gossip matrices \citep{dimakis2010synchgossip} or with randomized pairwise communications \citep{boyd2006gossip}). They originally were introduced to compute the global average of local vectors through peer-to-peer communication. Among the classical decentralized optimization algorithms, some alternate between gossip communications and local steps \citep{nedic2008distributed,koloskova2019decentralized,pmlr-v119-koloskova20a}, others use dual formulations and formulate the consensus constraint using gossip matrices to obtain decentralized dual or primal-dual algorithms \citep{scaman17optimal,hendrikx2019accelerated,even2021continuized,kovalev2021adom, alghunaim2019primaldual}, and benefit from natural privacy amplification mechanisms \citep{cyffers2022muffliato}.  Other approaches include non-symetric communication matrices \citep{Assran_2021} that are more scalable. We refer the reader to~\citet{nedic2018network} for a broader survey on decentralized optimization.  The works we relate to in this line of research are \citet{pmlr-v119-koloskova20a}, where a unified analysis of decentralized SGD is performed (the ‘‘gossip equivalent'' of our algorithm MC-SGD), and in particular contains rates for convex-non-smooth functions, and \citet{lian2019decentralizedmomemtum}, that performs an analysis of decentralized SGD with momentum in the smooth-non-convex case, which is the ‘‘gossip equivalent'' of our algorithm MC-SAG.

\subsection{Reinforcement Learning problems and online system identification}
In several applications (RL, time-series analysis \emph{e.g.}), a statistician may have access to values $(X_t)_{t\geq0}$ generated sequentially along the path of a Markov chain, observations from which she wishes to estimate a parameter For instance, \citet{kowshik2021streaming} consider a sequence of observations  $X_{t+1}=A_\star X_t+\xi_t$ for $\xi_t$ \emph{i.i.d.} centered noise, and $A_\star$ to estimate, and aim at finding $\hat A$ minimizing the MSE $\esp{\NRM{\sum_{t<T}(\hat{A}-A_\star)X}}$ where $X\sim\pi$ is the stationary distribution. 
Studying this problem under the lens of stochastic optimization, this boils down to building efficient strategis for SGD under Markov chain sampling, beyond the case of linear mean-squared regressions studied in \cite{kowshik2021streaming}.
While optimal offline policies have extensively been studied in this setting \citep{proutiere_sample,pmlr-v75-simchowitz18a_learning_without_mixing}, online algorithms that take the form of SGD-like algorithms have received little attention, and only focus on the case of quadratic losses with Markov chain as described above.
However, these analyses only focus on least squares and Markov chains of the form $X_{t+1}=A_\star X_t+\eta_t$. Under these specific assumptions, \citet{nagaraj_least_squares_markovian} prove that a dependency on $\tau_\mix$ for MC-SGD is inevitable, while using reverse-experince replay, \citet{kowshik2021streaming} breaks this and obtain sample-optimal online algorithms. Their algorithm however require to store a number of iterates that grow linearly with $\tau_\mix$.

The convergence guarantees we prove in the sequel for SGD under Markov chain sampling fit in this online framework, and refine previous analyses by removing strong regularity assumptions such as bounded iterates or bounded gradients \citep{tao2018mcsgd}, or strong assumptions on the Markovian structure data and least-squares problems \citep{nagaraj_least_squares_markovian,kowshik2021streaming}.

Finally, note that in our setting, the iterates of the algorithms considered (denoted as $(\xx_k)_{k\geq0}$) and the Markov chain $(v_k)_{k\geq 0}$ are dependent of each other. More precisely, $(v_k)_{k\geq0}$ is a Markov chain whose states do not depend on the iterate sequence, while $\xx_k$ is $(v_\ell)_{\ell\leq k-1}$-measurable. 
This setting is sometimes referred to as \emph{exogenous Markov noise} \cite{rust1986exogenous}.
Another line of works, pioneered by \cite{Benveniste1990}, considers Markov transitions for $(v_k)$ where  $v_{k+1}|v_k$ is sampled using a Markov transition kernel $P_{\xx_0,\ldots,\xx_{k-1}}$ that is directly linked to the iterates.
This orthogonal line of work of stochastic approximation with Markovian noise is related to sampling (through the MCMC algorithm), adaptive filtering, and other related problems that involve exploration \cite{Brown1985,Andrieu2005,Andrieu2006,Fort2016,blanke2023flex}.
Our work aims at finding precise rates of convergence as in the convex or non-convex optimization literature \cite{bubeck2014convex,carmon2021lower}, under the mildest assumptions on the exogenous Markov-chain $(v_k)_{k\geq0}$.

Before unrolling our results, we start by reminding definitions and results related to Markov chain theory in Section~\ref{sec:mc}, before presenting our contributions in Section~\ref{sec:related_motiv}

\section{Markov chains preliminaries}\label{sec:mc}

We refer the interested reader to \citet{LevinPeresWilmer2006} for a thorough introduction to Markov chain theory.
In this section, we define mixing, hitting and cover times for a Markov chain on a finite space set $\cV$ of cardinality $n$. However, note that these definitions can be extended to the more general setting where $\cV$ is infinite (either countable or not). 
We focus on Markov chain on finite state spaces, but note that the mixing time of a Markov chain can similarly be defined on inifinite state spaces (countable and continuous state spaces).
In this paper, all results that involve only the mixing time of the Markov chain (the results from Section~\ref{sec:mcsgd}) easily generalize to infinite state spaces.

\begin{definition}
    Let $P\in \R^{\cV\times\cV}$ be a stochastic matrix (\emph{i.e.}~$P_{v,w}\geq0$ for all $v,w\in\cV$ and $\sum_{w\in\cV}P_{v,w}=1$ for all $v\in\cV$).
    A time-homogeneous Markov chain on $\cV$ of transition matrix $P$ is a stochastic process $(X_t)_{t\geq 0}$ with values in $\cV$ such that, for any $t\geq 0$ and $w,v_0,\ldots,v_{t-1},v\in\cV$,
    \begin{equation*}
        \proba{X_{t+1}\!\!=w|X_t\!=\!v,X_{t-1}\!=\!v_{t-1},\ldots,X_0\!=\!v_0}=P_{v,w}\,.
    \end{equation*}
\end{definition}
A Markov chain of transition matrix $P$ is irreducible if, for any $v,w\in\cV$, there exists $t\geq0$ such that $(P^t)_{vw}>0$. A Markov chain of transition matrix $P$ is aperiodic if there exists $t_0>0$ such that for all $t\geq t_0$ and $v,w\in\cV$, $(P^t)_{v,w}>0$.
Any irreducible and aperiodic Markov chain on $\cV$ admits a stationary distribution $\pi$, that verifies $\pi P= \pi$. 
It finally holds that, if $P$ is reversible ($\pi_vP_{v,w}=\pi_wP_{w,v}$ for all $v,w\in\cV$), denoting as $\lambda_P=1-\max_{\lambda\in{\rm Sp}(P)\setminus \set{1}}|\lambda|>0$ the absolute spectral gap of $P$, where ${\rm Sp}(P)$ is the spectrum of $P$, for any stochastic vector $\pi_0\in\R^\cV$:
\footnote{we write $\NRM{z}_\pi^2=\sum_{v\in\cV}\pi_vz_v^2$ for $z\in\R^\cV$, and $\NRM{\cdot}$ always stands for the Euclidean norm}
\begin{equation*}
    \NRM{\pi_0 P^t -\pi}_\pi\leq (1-\lambda_P)^t\NRM{\pi_0-\pi}_\pi\,.
\end{equation*}
If the chain is not reversible, there is still a linear decay, but in terms of total variation distance rather than in the norm $\NRM{\cdot}_\pi$ (Chapter 4.3 of \citet{LevinPeresWilmer2006}).
In the sequel, $(v_t)_{t\geq0}$ is any irreducible aperiodic Markov chain of transition matrix $P$ on $\cV$ of stationary distribution $\pi$ (not necessarily the uniform distribution on $\cV$).

Furthermore, we define the graph $G=(\cV,\cE)$ over the state space $\cV$ through the relation $\edgevw\in\cE \iff P_{v,w}>0$ for $v$ and $w$ two distinct states.
Consequently, the Markov chain $(v_t)_t$ can also be seen as a random walk on graph $G$, with transition probability $P$.
In the random walk decentralized optimization case, this graph coincides with the communication graph.
In the sequel, for $t\geq0$ and $v\in\cV$, $\esp{\cdot|v_t=v}$ and $\proba{\cdot|v_t=v}$ respectively denote the expectation and probability conditioned on the event $v_t=v$.
Similarly for $\pi_t$ a probability distribution on $\cV$, $\esp{\cdot|v_t\sim\pi_t}$ and $\proba{\cdot|v_t\sim\pi_t}$ refers to conditioning on the law of $v_t$.

\begin{definition}[Mixing, hitting and cover times]\label{def:times}
    For $w\in\cV$, let $\tau_w=\inf\set{t\geq1\,|\, v_t=w}$ be the time the chain reaches $w$ (or returns to $w$, in the case $v_0=w$). We define the following quantities.
    \begin{enumerate}
        \item \emph{Mixing time.} For $\eps>0$, the mixing time $\tau_\mix(\eps)$ of $(v_t)$ is defined as, where $\dd_{\rm TV}$ is the total-variation distance:
        \begin{equation*}
            \tau_\mix(\eps)=\inf\set{t\geq1\,|\, \forall \pi_0\,, \dd_{\rm TV}(P^t\pi_0,\pi)\leq \eps}\,,
        \end{equation*}
        and we define the mixing time $\tau_\mix$ as $\tau_\mix=\tau_\mix(\pi_{\min}/2)$ 
        \footnote{this definition of mixing time is not standard: \citet{LevinPeresWilmer2006} define it as $\tau_\mix(1/4)$, \citet{walkman} define it as we do; however, as explained in Chapter 4.5 of \citet{LevinPeresWilmer2006}, these definitions are equivalent up to a factor $\ln(1/\pi_{\min})$}
        where $\pi_{\min}=\min_{v\in\cV}\pi_v$.
        \item \emph{Hitting and cover times.} The hitting time $\tau_\hit$ and cover time $\tau_\cov$ of $(v_t)$ are defined as:
        \begin{equation*}
            \begin{aligned}
                \tau_\hit&=\max_{(v,w)\in\cV^2} \esp{\tau_w|v_0=v}\,,\\
                \tau_\cov&=\max_{v\in\cV}\esp{\max_{w\in\cV}\tau_w|v_0=v}\,.
            \end{aligned}
        \end{equation*}
    \end{enumerate}
    The mixing time is the number of steps of the Markov chain required for the distribution of the current state to be close to the stationary probability $\pi$.
    Starting from any arbitrary $v_0$, the hitting time bounds the time it takes to reach any fixed $w$, while the cover time bounds the number of steps required to visit all the nodes in the graph.
\end{definition}

Note that if the chain is reversible, $\tau_\mix(\eps)$ is closely related to $\lambda_P$ through $\tau_\mix(\eps)\leq \lceil \lambda_P^{-1}\ln(\pi_{\min}^{-1}\eps^{-1})\rceil$. Under reversibility assumptions, we defined the \emph{relaxation time} of the Markov chain as $\tau_{\rm rel}=1/\lambda_P$.
More generally without reversibility, $\tau_\mix(\eps)\leq \lceil\log_2(1/\eps)\rceil \tau_\mix(1/4)$.
Then, as we prove in Appendix~\ref{app:basic}, $\tau_\hit$ always satisfies $\tau_\hit\leq 2\pi_{\min}^{-1}\tau_\mix$.
Finally, using \citet{matthews-covering}' method (detailed in Chapter 11.4 of \citet{LevinPeresWilmer2006}), we have $\tau_\cov\leq \ln(n)\tau_\hit$.

\section{
Contributions}\label{sec:related_motiv}


In our paper, we analyze theoretically stochastic gradient methods with Markov chain sampling (such as MC-SGD in Equation~\eqref{eq:mcsgd}), and aim at deriving complexity bounds under the mildest assumptions possible.
We first derive in Section~\ref{sec:lower} complexity lower bounds for such methods, making appear $\tau_\hit$ as the Markov chain quantity that slows down such algorithms.

We then study MC-SGD under various regularity assumptions in Section~\ref{sec:mcsgd}: we remove the bounded gradient assumption of all previous analyses, obtain rates under a $\mu$-PL assumption, and prove a linear convergence in the interpolation regime, where noise and function dissimilarities only need to be bounded at the optimum.

In the data-heterogeneous setting (functions $f_v$ that can be arbitrarily dissimilar) and in the case where $\cV$ (the state space of the Markov chain) is finite, we introduce MC-SAG in Section~\ref{sec:mcSAG}, a variance-reduced alternative to MC-SGD, that is perfectly suited to decentralized optimization. Using time adaptive stepsizes, this algorithm has a rate of convergence of $\tau_\hit/T$ and thus matches that of our lower bound, up to acceleration.

We discuss in Section~\ref{sec:discussion} the implications of our results. 
In particular, we prove that random-walk based decentralization is more communication efficient than consensus-based approaches; prior to our analysis, this was only shown empirically \citep{walkman,johansson2010randomizedincremental}.
Further, our results formally prove that using all gradients along the Markov chain trajectory leads to faster rates; as in the previous case, this was only empirically observed before \citep{tao2018mcsgd}.
These two consequences are derived from the fact that MC-SAG depends only on $\tau_\hit$ rather than the traditionally used quantity $n\tau_\mix$, that can be arbitrarily bigger (Table~\ref{tb:times}).

\section{Oracle complexity lower bounds under Markov chain sampling
}\label{sec:lower}

In this section, we provide oracle complexity lower bounds for finding stationary points of the function $f$ defined in~\eqref{eq:f}, for a class of algorithms that satisfy a ‘‘Markov sampling scheme''.
For a given Markov chain $(v_t)$ on $\cV$, we consider algorithms verifying the following procedural constraints, for some fixed initialization $\cM_0=\set{\xx_0}$ an then for $t\geq0$,
\begin{enumerate}
    \item A iteration $t$, the algorithm has access to function $f_{v_t}$ and may extend its memory:
    \begin{equation*}
        \cM_{t+1}={\rm Span}(\set{\xx\,,\,\nabla f_{v_t}(\xx)\,,\quad \xx\in\cM_{t}})\,.
    \end{equation*}
    \item \textbf{Output:} the algorithm specifies an output value $\xx_t\in\cM_t$.
\end{enumerate}
We call algorithms verifying such constraints ‘‘black box procedures with Markov sampling~$(v_t)$''.
Such procedures as well as the result below are inspired by the distributed black-box procedures defined in~\citet{scaman17optimal}.
We use the notation $a(\cdot)=\Omega(b(\cdot))$ for $\exists C>0$ such that $a(\cdot)\geq C b(\cdot)$ in the theorem below, and classically consider the limiting situation $d\to\infty$, by assuming we are working in $\ell_2=\set{(\theta_k)_{k\in\N}\in\R^\N:\sum_k\theta_k^2<\infty}$.
\begin{theorem}\label{thm:lower}
    Assume that $\tau_v$ (see Definition \ref{def:times}) has finite second moment for any $v\in\cV$.
	Let $\Delta,B>0$, $L>0$ and $\mu>0$, denote $\kappa=L/\mu$.
    Let $\xx_0$ be fixed.
	\begin{enumerate}
		\item \emph{Non-convex lower bound:} there exist functions $(f_v)_{v\in\cV}$ such that $f=\sum_{v\in\cV}\pi_v f_v$ is $L$-smooth, and $f(\xx_0)-\min_\xx f(\xx)\leq \Delta$ and such that for any $T$ and any Markov black-box algorithm that outputs $\xx_T$ after $T$ steps, we have:
		\begin{equation*}
			\NRM{\nabla f(\xx_T)}^2= \Omega\left(L\Delta\left(\frac{\tau_{\rm hit}}{T}\right)^2\right)\,.
		\end{equation*}
		\item \emph{Convex lower bound:} there exist functions $(f_v)_{v\in\cV}$ such that $f=\sum_{v\in\cV}\pi_v f_v$ is convex and $L$-smooth and minimized at some $\xx^\star$ that verifies $\NRM{\xx^0-\xx^\star}^2\leq B^2$, and such that for any $T$ and any Markov black-box algorithm that outputs $\xx_T$ after $T$ steps, we have:
		\begin{equation*}
			f(\xx_T)-f(\xx^\star)= \Omega\left(LB^2\left(\frac{\tau_{\rm hit}}{T}\right)^2\right)\,.
		\end{equation*}
		\item \emph{Strongly convex lower bound:} there exist functions $(f_v)_{v\in\cV}$ such that $f=\sum_{v\in\cV}\pi_v f_v$ is $\mu$-strongly convex and $L$-smooth and minimized at some $\xx^\star$ that verifies $\NRM{\xx^0-\xx^\star}^2\leq B^2$, and such that for any $T$ and any Markov black-box algorithm that outputs $\xx^T$ after $T$ steps, we have:
		\begin{equation*}
			f(\xx_T)-f(\xx^\star)=\Omega\left(LB^2\exp\left(-\frac{T}{\sqrt\kappa\tau_{\rm hit}}\right)\right)\,.
		\end{equation*}
	\end{enumerate}
\end{theorem}
A complete proof can be found in Appendix~\ref{app:lower}. 
The hitting time of the Markov chain bounds, starting from any point in $\cV$, the mean time it takes to reach any other state in the graph.
Making no other assumptions than smoothness, having rates that depend on this hiting time is thus quite intuitive.

\section{Analysis of Markov-Chain SGD}\label{sec:mcsgd}


We have shown in last subsection that, in order to reach an $\eps$-stationary point with Markov sampling, the optimizer is slowed down by the hitting time of the Markov chain; this lower bound being worst-case on the functions $(f_v)$, we here add additional similarity assumptions, that are still milder than classical ones in this setting \cite{tao2018mcsgd}.
Studying the iterates generated by~\eqref{eq:mcsgd}, we obtain in this section a dependency on $\tau_\mix$, provided bounded gradient dissimilarities (Assumptions~\ref{hyp:dissimilarities} and~\ref{hyp:local-sto}). 

We here assume that $(v_t)_{t\geq0}$ is a Markov chain on $\cV$ of invariant probability $\pi$ (not necessarily the uniform measure on $\cV$).
Our analysis of Markov chain SGD \textbf{does not} rely on finite state spaces: $\cV$ is not assumed to be finite (it can be any infinite countable, or continuous space).
In this section, the function $f$ studied is defined as $$f(\cdot)=\E_{v\sim\pi}[f_v(\cdot)]\,,$$ as in~\eqref{eq:f_gen}.
Consequently, for the MC-SGD algorithm for decentralized optimization over a given graph $G$ to minimize the averaged function over all nodes (as in \eqref{eq:f}), $\pi$ needs to be the uniform probability over $\cV$.

We first derive convergence rates under smoothness assumptions with or without a $\mu$-PL inequality that holds, before improving our results under strong convexity assumptions, under which we prove a linear convergence rate in the interpolation regime.
We finally add local noise (due to sampling, or additive gaussian noise to enforce privacy) in the final paragraph of this Section.

\subsection{Analysis under bounded gradient dissimilarities}

\begin{assumption}\label{hyp:dissimilarities}
	There exists $(\sigma^2_{v})_{v\in\cV}$ such that for all $v\in\cV$ and all $\xx\in\R^d$, we have\footnote{this assumption could be replaced by a more relaxed noise assumption of the form $\NRM{\nabla f_v(\xx)-\nabla f(\xx)}^2\leq M\NRM{\nabla f(\xx)}^2+\sigma^2_v$}:
	\begin{equation*}
		\NRM{\nabla f_v(\xx)-\nabla f(\xx)}^2\leq \sigma^2_v\,,
	\end{equation*}
	and we denote $\bar\sigma^2=\E_{v\sim\pi}\left[\sigma_v^2\right]$ and $\sigma_{\max}^2=\max_{v\in\cV}\sigma_v^2$.
\end{assumption}

\begin{assumption}\label{hyp:smooth}
    Each $f_v$ is $L$-smooth, $f$ is lower bounded, its minimum is attained at some $\xx^\star\in\R^d$.
\end{assumption}


\begin{theorem}[MC-SGD] \label{thm:mcsgd}
	Assume that Assumptions~\ref{hyp:dissimilarities} and \ref{hyp:smooth} hold, and let $\Delta\geq f(\xx_0)-f(\xx^\star) + \sigma_{\max}^2/L$.
	\begin{enumerate}
	    \item \label{thm:mcsgd_nonsmooth} For a constant time-horizon dependent step size $\gamma$ (\emph{i.e.}, $\gamma$ is a functiion of $T$), the iterates generated by Equation~\eqref{eq:mcsgd} satisfy, for $T\geq 2\tau_{\mix}\ln(\tau_\mix)$:
    \footnote{$\Tilde\cO$ hides logarithmic factors}
	\begin{equation*}
		\E\NRM{\nabla f(\hat\xx_T)}^2\!=\!\Tilde\cO\!\left(\frac{\Delta L\tau_{\rm mix}}{T}+\frac{\sqrt{L\Delta\bar\sigma^2\tau_{\rm mix}}+\bar\sigma^2}{\sqrt{T}}\right)\,,
	\end{equation*}
	where $\hat \xx_T$ is drawn uniformly at random amongst $\xx_0,\ldots,\xx_{T-1}$.
	\item \label{thm:mcsgd_PL} If $f$ additionally verifies a $\mu$-PL inequality (for any $\xx\in\R^d$, $\NRM{\nabla f(\xx)}^2)\geq 2\mu(f(\xx)-f(\xx^\star))$), for a constant time-horizon dependent step size $\gamma$, the iterates generated by Equation~\eqref{eq:mcsgd} satisfy, for $T\geq 2\tau_{\mix}\ln(\tau_\mix)$ a numerical constant $c>0$, and $\kappa=L/\mu$, with $F_T=\esp{f(\xx_T)-f(\xx^\star)}$:
	\begin{equation*}
		F_T\leq e^{-\frac{cT}{\kappa \tau_\mix\ln(T)}}\Delta + \Tilde\cO\left(\frac{\tau_\mix\bar\sigma^2}{\mu T}\right)\,,
	\end{equation*}
	\end{enumerate}
\end{theorem}

Theorem~\ref{thm:mcsgd} is proved in Appendix~\ref{app:mcsgd}, by enforcing a delay of order $\tau_\mix$ and relying on recent analyses of delayed SGD and SGD with biased gradients. 
As explained in the introduction, removing the bounded gradient assumption present in previous works \citep{johansson2010randomizedincremental,tao2018mcsgd,duchi2012ergodic} that study Markov chain SGD (in the mirror setting, or with subdifferentials), and replacing it by a much milder and classical assumption of bounded gradient dissimilarities \citep{scaffold}, we thus still managed to obtain similar rates.
Further, if $f$ verifies a $\mu$-PL inequality (if for any $\xx\in\R^d$, $\NRM{\nabla f(\xx)}^2)\geq 2\mu(f(\xx)-f(\xx^\star))$), we have an almost-linear rate of convergence:
this is the first rate under $\mu$-PL or strong convexity assumptions for MC-SGD-like algorithms, that we even refine further in next subsection.

\subsection{Tight rates and linear convergence in the interpolation regime}


We now study MC-SGD under the following assumptions, to derive faster rates, that only depend on the sampling noise at the optimum. The interpolation regime -- often related to overparameterization -- refers to the case where there exists a model $\xx^\star\in\R^d$ minimizing all $f_v$ for $v\in\cV$, leading to $\sigma_\star^2=0$ in Assumption~\ref{hyp:noise_opti}, and to a linear convergence rate below.

\begin{assumption}\label{hyp:local-sto}
Functions $f_v$ are 
$L$-smooth and $\mu$-strongly convex. We denote $\kappa=L/\mu$.
\end{assumption}
\begin{assumption}[Noise at the optimum]\label{hyp:noise_opti}
    Let $\xx^\star$ be a minimizer of $f$. We assume that for some $\sigma_\star\geq0$, we have for all $v\in\cV$:
    \begin{equation*}
        \NRM{\nabla f_v(\xx^\star)}^2\leq \sigma_\star^2\,.
    \end{equation*}
\end{assumption}
\begin{theorem}[Unified analysis]\label{thm:mcsgd_interpolation}
    Under Assumptions~\ref{hyp:local-sto} and \ref{hyp:noise_opti}, the sequence generated by \eqref{eq:mcsgd} satisfies, if $\gamma L<1$:
    \begin{equation*}\begin{aligned}
        &\esp{\NRM{\xx_T-\xx^\star}^2}\leq 2(1-\gamma\mu)^T\NRM{\xx_0-\xx^\star}^2 \\
        &+2\frac{\gamma^3TL}{\mu}\sum_{0\leq s\leq T}(1-\gamma\mu)^{T-s}\esp{\NRM{\sum_{s\leq t<T}\nabla f_{v_t}(\xx^\star)}^2}\,.
        \end{aligned}
    \end{equation*}
    In the interpolation regime, $\nabla f_v(\xx^\star)=0$ for all $v\in\cV$, so that:
    \begin{equation*}
        \esp{\NRM{\xx_T-\xx^\star}^2}\leq 2(1-\gamma\mu)^T\NRM{\xx_0-\xx^\star}^2\,.
    \end{equation*}
\end{theorem}
The result in Theorem~\ref{thm:mcsgd_interpolation} is in fact true irrespectively of the sequence $(v_t)$ chosen: it does not require $(v_t)$ to specifically be a Markov chain. This property is used in the next Corollary, that also highlights the fact that by studying distance to the optimum, a condition number is lost in the process. This is the case in many previous analyses of other different algorithms (\emph{e.g.}, Bregman/Mirror-SGD \citep{dragomir2021fast} or SGD with random-resfhuffling \citep{mischenko2020randomreshuffling}, which is in fact a particular instance of MC-SGD, that our analysis recovers), that study distances to the optimum (with respect to some mirror map, in the case of Mirror SGD), and therefore obtain an extra $\kappa$ factor in the noise term.
Theorem~\ref{thm:mcsgd_interpolation} is proved by generalizing the proof technique of \cite{mischenko2020randomreshuffling} to arbitrary orderings and for unbounded time horizons.

\begin{remark}[Random resfhuffling]
    A special case of Theorem~\ref{thm:mcsgd_interpolation} is SGD with random reshuffling.
    By analyzing SGD with random-reshuffling as SGD with a Markovian ordering (on an extended state space), Theorem~\ref{thm:mcsgd_nonsmooth},\ref{thm:mcsgd_PL} also recover rates for SGD with random reshuffling for which we have $\tau_\mix=n$.
    Moreover, since Theorem~\ref{thm:mcsgd_interpolation} generalizes Theorem 1 of \cite{mischenko2020randomreshuffling}, we also recover their rate as a special case by bounding each term $\esp{\NRM{\sum_{s\leq t<T}\nabla f_{v_t}(\xx^\star)}^2}$.
\end{remark}

We specify Theorem~\ref{thm:mcsgd_interpolation} under a Markovian sampling scheme in next corollary: the noise term at the optimum takes the form $\tau_\mix/T$.

\begin{corollary}[MC-SGD, interpolation]\label{thm:mcsgd_interpolation_step}
    Under Assumptions~\ref{hyp:local-sto} and \ref{hyp:noise_opti}, for $T\geq1$ and for a well chosen stepsize $\gamma>0$, the iterates generated by~\eqref{eq:mcsgd} satisfy:
    \begin{equation*}
    \begin{aligned}
        \esp{\NRM{\xx_T-\xx^\star}^2}&\leq 2e^{-\frac{T}{\kappa}}\NRM{\xx_0-\xx^\star}^2 \\ &+\Tilde\cO\left(\frac{L\tau_\mix\big(\frac{1}{4}\big)\sigma_\star^2}{\mu^3T}\right)\,.
        \end{aligned}
    \end{equation*}
\end{corollary}

This result is stronger than Theorem~\ref{thm:mcsgd}.\ref{thm:mcsgd_PL}, for \emph{(i)} noise amplitude and gradient dissimilarities only need to be bounded at the optimum;
\emph{(ii)} the ‘‘optimization term'' (the first one) is not slowed down by the mixing time. 
This comes at the cost of strong convexity assumptions, stronger than a $\mu$-PL inequality for $f$.
The term $\frac{\tau_\mix\sigma_\star}{T}$ cannot be removed in the general case, as next proposition shows. Hence, since the two other terms have optimal dependency in terms of Markov-chain and noise related quantities, our analysis ends up being sharp. 

Corollary~\ref{thm:mcsgd_interpolation} together with the following proposition are an extension of \citet{nagaraj_least_squares_markovian}, who proved similar results for MC-SGD with constant stepsize on least square problems on Markovian data of a certain form (for linear online system identification).

\begin{proposition}\label{prop:lower_mcsgd}
For any $\cV$ (such that $|\cV|\geq 2$) and $\tau>1$, there exists a Markov chain on $\cV$ of relaxation time $\tau$, functions $(f_v)_{v\in\cV}$ and $\xx_0\in\R^d$ such that given any stepsize $\gamma$, the iterates of Equation~\eqref{eq:mcsgd} output $\xx_T$ for any $T>0$ verifying $\NRM{\xx_T-\xx_0}^2=\Tilde\Omega(\tau \sigma_\star^2/T)$, and the assumptions of Theorem~\ref{thm:mcsgd_interpolation} hold.
\end{proposition}


\subsection{MC-SGD with local noise}\label{sec:localnoise}

In the two previous subsections, we analyzed SGD with Markovian sampling schemes, where the stochasticity only came from the Markov chain $(v_k)_{k\geq0}$.
We now generalize the analysis and results to SGD with both Markovian sampling, and local noise, by studying the sequence:
\begin{equation}\label{eq:localnoise}
    \xx_{t+1}=\xx_t - \gamma_t \gg_t\,.
\end{equation}
We now formulate the form stochastic gradients $\gg_t$ can take.

\begin{assumption}\label{hyp:local_noise}
    For all $v\in\cV$, the function $f_v$ satisfies $f_v(\xx)=\esp{F_v(\xx,\xi_v)}$ for all $\xx\in\R^d$, where $\xi_v\sim\cD_v$.
    Furthermore, there exists a Markov-chain $(v_t)_{t\geq0}$ such that for all $t\geq 0$, \[\gg_t= \nabla_\xx F_{v_t}(\xx_t,\xi_t)\,,\] where $\xi_t\sim \cD_{v_t} | v_t$ is independent from $v_0,\ldots,v_{t-1}$ and $\xi_0,\ldots,\xi_{t-1}$.
\end{assumption}

A direct consequence of Assumption~\ref{hyp:local-sto} is that $\esp{\gg_t|\xx_t,v_t}=\nabla f_{v_t}(\xx_t)$.
Two main applications of Assumption~\ref{hyp:local-sto} are:
\begin{enumerate}
    \item \textbf{Local sampling.} If $f_v(\xx)=\frac{1}{m}\sum_{i=1}^m f_{v,i}(\xx)$ (agent $v$ has $m$ local samples), agent $m$ may use only a batch $\cB\subset[m]$ of its samples, leading to stochastic gradients $\gg_t$ in \eqref{eq:localnoise} of the form:
    \begin{equation*}
        \gg_t= \frac{1}{|\cB_t|}\sum_{i\in\cB_t}\nabla f_{v_t,i}(\xx_t)\,,
    \end{equation*}
    for random batches $(\cB_t)_{t\geq0}$.
    \item \textbf{Differential privacy.} Adding local noise (\emph{e.g.}, additive Gaussian random noise) enforces differential privacy under suitable assumptions.
    A private decentralized token algorithm is then \textbf{Differentially Private MC-SGD} (DP-MC-SGD), with iterates \eqref{eq:localnoise} where $\gg_t$ satisfies
    \begin{equation}
        \gg_t= \nabla f_{v_t}(\xx_t) + \eta_t\,,
    \end{equation} 
    where $(v_t)$ is the Markov chain (random walk performed by the token on the communication graph), and $\eta_t\sim\cN(0,\sigma_t^2 I_d)$ is sampled independently from the past, to enforce differential privacy.
\end{enumerate}

Under Assumption~\ref{hyp:local-sto}, a direct generalization of Theorem~\ref{thm:mcsgd_interpolation} and Corollary~\ref{thm:mcsgd_interpolation_step} is the following.

\begin{theorem}[MC-SGD with local noise]\label{thm:local_noise}
    Assume that Assumptions~\ref{hyp:local_noise},\ref{hyp:noise_opti} holds, each $F_v(\cdot,\xi)$ is $\mu$-strongly convex $L$-smooth, and there exists $\sigma_{\rm local}^2$ such that:
    \begin{equation*}
        \esp{\NRM{\gg_t-\nabla f_{v_t}(\xx_t)}^2|\xx_t,v_t}\leq \sigma^2_{\rm local}\,.
    \end{equation*}
    Then, for a well chosen $\gamma>0$, the iterates generated by \eqref{eq:localnoise} satisfy:
    \begin{align*}
        \esp{\NRM{\xx_T-\xx^\star}^2}&\leq 2e^{-\frac{T}{\kappa}}\NRM{\xx_0-\xx^\star}^2 \\ 
        &+\Tilde\cO\left(\frac{L}{\mu^3T}\big(\sigma_{\rm local}^2+\tau_\mix\big(\frac{1}{4}\big)\sigma_\star^2 \big)\right)\,.
    \end{align*}
\end{theorem}

Importantly, and as one would have expected, local noise is \textbf{not} impacted by the mixing time of the underlying random walk.
While we did not pursue in this direction, this observation could easily be made under other regularity assumptions, and such a result would hold for instance under the assumptions of Theorem~\ref{thm:mcsgd_nonsmooth} or \ref{thm:mcsgd_PL}.
While Differentially Private MC-SGD sounds appealing for performing decentralized and differentially private optimization, we here only provided a utility analysis, the privacy analysis being left for future works.

\section{Analysis of Markov-Chain SAG}\label{sec:mcSAG}

\begin{algorithm}[h]
    \caption{Markov Chain SAG (MC-SAG)}
    \label{algo:mcSAG}
    \begin{algorithmic}[1]
    \STATE \textbf{Input:} 
    $\xx_0\in \R^{d}$, $\hh_v\in\R^d$ for $v\in\cV$ and $\bar\hh_0\in\R^d$, stepsizes $\gamma_t>0$, $v_0\in\cV$
    \FOR{$t = 0,1,\dots$}
        \STATE 
        Compute $\nabla f_{v_t}(\xx_t)$
        \STATE $\bar\hh_{t+1}=\bar\hh_t+\frac{1}{n}\big(\nabla f_{v_t}(\xx_t)-\hh_{v_t}\big)$
        \STATE $\xx_{t+1}=\xx_t - \gamma_t\bar\hh_{t+1}$
        \STATE $\hh_{v_t}\longleftarrow\nabla f_{v_t}(\xx_t)$
        \STATE Sample $v_{t+1}\sim P_{v_t,\cdot}$
     \ENDFOR
    \end{algorithmic}	
\end{algorithm}

After providing convergence guarantees for the most natural algorithm (MC-SGD) under a Markov chain sampling on the set $\cV$, we prove that one can achieve a rate of order $1/T$ (rather than the $1/\sqrt{T}$ previously obtained) in the smooth setting, by applying the variance reduction techniques present in \citet{SAG}, that first introduced the Stochastic Averaged Gradient algorithm, together with a time-adaptive stepsize policy described below.
Our faster rate with variance reduction leads of a dependency on $\tau_\hit$ instead of $\tau_\mix$; since we do not make any other assumption other than smoothness, this is unavoidable in light of our lower bound (Theorem~\ref{thm:lower}).

\paragraph{MC-SAG}
The MC-SAG algorithm is described in Algorithm~\ref{algo:mcSAG}.
The recursion leading to the iterate $\xx_t$ can then be summarized as, for stepsizes $(\gamma_t)_{t\geq0}$, under the initialization $\hh_v=\nabla f_v(\xx_0)$ and $\bar \hh=\nabla f(\xx_0)$:
\begin{equation}\label{eq:mcSAG}
    \xx_{t+1}=\xx_t -\frac{\gamma_t}{n}\sum_{v\in\cV} \nabla f_v(\xx_{d_v(t)})\,,
\end{equation}
where for $v\in\cV$, we define $d_v(t)=\sup\set{s\leq t\,|\, v_s=v}$ as the last previous iterate at which $v$ was the current state of the Markov chain. By convention, if the set over which the supremum is taken is empty, we set $d_v(t)=0$.
We handle both the initialization described just above for $\hh_v,\bar\hh$ and arbitrary initialization in our analysis below.

In the same way that MC-SGD reduces to vanilla SGD if $(v_t)$ is an \emph{i.i.d.} uniform sampling over $\cV$, MC-SAG boils down to the SAG algorithm \citep{SAG} in that case and under the initialization $\hh_v=\nabla f_v(\xx_0)$ and $\bar \hh=\nabla f(\xx_0)$.
In a decentralized setting, nodes keep in mind their last gradient computed (variable $\hh_v$ at node $v$). At all times, $\bar\hh_t$ is an average of these $\hh_v$ over the graph, and is, in the same way as $\xx_t$, updated along the random walk.
The MC-SAG algorithm is thus perfectly adapted to decentralized optimization.

\paragraph{Time-adaptive stepsize policy}
To obtain our convergence guarantees, a time-adaptive stepsize policy $(\gamma_t)$ is used, as in Asynchronous SGD \citep{mischenko_asynchronous_sgd} to obtain delay-independent guarantees. For $t\geq0$, let the stepsize $\gamma_t$ be defined as:
\begin{equation}\label{eq:stepsizes}
    \gamma_t=  \frac{1}{2L\big(\tau_\hit + \max_{v\in\cV}(t-d_v(t))\big)}\,.
\end{equation}
Denoting $\tau_t=\max_{v\in\cV}(t-d_v(t))$, this quantity can be tracked down during the optimization process. Indeed, if agent $v_t$ receives $\tau_{t-1}$ together with $(\xx_t,\bar\hh_t)$, she may compute $\tau_{t}$ as:
\begin{equation*}
    \tau_t= \max\big(\tau_{t-1}+1\,,\, t-d_{v_t}(t)\big)\,,
\end{equation*}
where $t-d_{v_t}(t)$ is the number of iterations that took place since the last time the Markov chain state was $v_t$.
Hence, if agents keep track of the number of iterations, the adaptive stepsize policy \eqref{eq:stepsizes} can be used in Algorithm~\ref{algo:mcSAG}, as long as agent $v_t$ sends $(\tau_t,t)$ to $v_{t+1}$, yielding the following result.

We now present the convergence results for MC-SAG. 
$(v_t)$ is in this section assumed to be a Markov chain on $\cV$ of finite hitting time $\tau_\hit$. Importantly, the next Theorem does not require any additional assumption on $(v_t)$ such as reversibility, or even that it has a stationary probability that is the uniform distribution: the non-symmetric but easily implementable transition probabilities $P_{v,w}=1/(d_v+1)$ for $w=v$ or $w\sim v$ can be used here, as well as non-reversible random walks than can have much smaller mixing and hitting times. 
The function $f$ studied is here independent of the Markov chain, and is defined as in \eqref{eq:f}, the uniformly averaged function over all states $v\in\cV$ (or over all agents in the network).

\begin{theorem}[MC-SAG]\label{thm:mcSAG_varying}
	Assume that Assumption~\ref{hyp:smooth} holds and that the Markov chain has a finite hitting time (for an arbitrary invariant probability).
	\begin{enumerate}
	    \item \label{thm:mcSAG_varying_init} Under the initialization: 
	    \begin{equation*}
	        \begin{aligned}
	            \hh_v&=\nabla f_v(\xx_0)\,,\\
	            \bar\hh&=\nabla f(\xx_0)\,,
	        \end{aligned}
	    \end{equation*}
     using the adaptive stepsize policy defined in Equation~\eqref{eq:stepsizes}, the sequence generated by Algorithm~\ref{algo:mcSAG} satisfies, for any $T>0$:
	\begin{equation*}
		\esp{\min_{t<T}\NRM{\nabla f(\xx_t)}^2}\leq 8L\big(f(\xx_0)-f(\xx^\star)\big)\frac{\tau_\hit\ln(n)}{T}\,.
	\end{equation*}
    	  
    	  \item \label{thm:mcSAG_varying_arbitrary} Under any arbitrary initialization that satisfies $\bar\hh=\frac{1}{n}\sum_{v\in\cV} \hh_v$,
     using the adaptive stepsize policy defined in Equation~\eqref{eq:stepsizes}, the sequence generated by Algorithm~\ref{algo:mcSAG} satisfies, for any $T>0$:
	\begin{equation*}
		\esp{\min_{t<T}\NRM{\nabla f(\xx_t)}^2}\leq 16L\Delta\frac{\tau_\hit\ln(n)}{T}\,,
	\end{equation*}
    where \[\Delta=f(\xx_0)-f(\xx^\star)+\frac{1}{8n}\sum_{v\in\cV}\NRM{\nabla f_v(\xx_0)-\hh_v}^2\,.\]
    \end{enumerate}
\end{theorem}

Theorem~\ref{thm:mcSAG_varying} is proved in Appendix~\ref{thm:mcSAG_varying}.
Up to the logarithmic factor in $n$, the rates in Theorem~\ref{thm:mcSAG_varying} are the non-accelerated versions of the lower-bound in Theorem~\ref{thm:lower}.



\section{Discussion of our results}\label{sec:discussion}

\subsection{Communication efficiency: comparison of our results with consensus-based approaches}

We summarize the communication efficiencies in Table~\ref{tb:rates} (in terms of total number of communications required to reach an $\eps$-stationary point), of classical gossip-based decentralized gradient methods (non-accelerated, since no accelerated method is known under our regularity assumptions).
We consider the algorithm of \cite{lian2019decentralizedmomemtum} (decentralized SGD with momemtum, state of the art decentralized gossip-based algorithm for this problem) with fixed communication matrix $W$ on the graph $G$  
together with the Walkman algorithm \cite{walkman} and our algorithms, for a Markov chain with transition matrix $P$.
For the sake of comparison, we take as gossip matrix $W=P$.
Consequently as shown in Table~\ref{tb:rates}, \emph{our algorithm (MC-SAG) always outperforms non-accelerated gossip-based decentralized gradient descent algorithms in terms of number of communications required to reach $\eps$-stationary points}.
Note that we do not claim the ‘‘overall superiority'' of our approach over classical decentralized optimization algorithms (the latter benefit from parallelization while ours do not), but a superiority in terms of communication efficiency.  

\begin{table}[h]
	\caption{
        Number of communications required (\# comm. below) to obtain an $\eps$-stationary point $\xx$ (verifying $\NRM{\nabla f(\xx)}^2\leq \eps$). 
        Logarithmic/constant factors hidden.
    }\label{tb:rates}
	\centering
	\begin{threeparttable}
		\begin{tabular}{c c c c c}
			\toprule[.1em]
			\begin{tabular}{c} \end{tabular} & {\color{blue} A}  &  {\color{blue} B}  & Our work \\
			\midrule
			\# comm.
			& $\eps^{-1}|\cE|\tau_{\rm mix}$
			& $\eps^{-1}n\tau_{\rm mix}^2$ & \begin{tabular}{c} $ \eps^{-2}\tau_\mix$  $^{\color{blue} c}$\\ $\eps^{-1}\tau_{\rm hit}$ $^{\color{blue} d}$\end{tabular} & \\ \cmidrule{1-1}
		\end{tabular} 
	\begin{tablenotes}
		{\scriptsize 
		\item {\color{blue} A}: \cite{lian2019decentralizedmomemtum}.
		\item {\color{blue} B}: \cite{walkman}.
		\item $^{\color{blue} c}$ MC-SGD under Assumption \ref{hyp:dissimilarities}. 
		\item $^{\color{blue} d}$ MC-SAG.
	    }
	\end{tablenotes}  		
	\end{threeparttable}
\end{table} 
\begin{table}[h]
	\caption{
        Hitting and mixing times of some known graphs, for the simple random walk. 
    }\label{tb:times}
	\centering
	\begin{threeparttable}
		\begin{tabular}{c  c c c}
			\toprule[.1em]
			\begin{tabular}{c} \end{tabular} & Cycle  & $d$-dim. torus & Complete graph \\
			\midrule
			\begin{tabular}{c} $\tau_\hit$ \end{tabular} &  $\cO(n^2)$  &  $\cO(n^{1+\frac{1}{d}})$ & $\cO(n)$\\
			\midrule
			\begin{tabular}{c} $n\tau_\mix$ \end{tabular} &  $\cO(n^3)$   &$\cO(n^{1+\frac{2}{d}})$& $\cO(n)$ \\ \cmidrule{1-1}
		\end{tabular} 
	\end{threeparttable}
\end{table} 

The dependency on the quantity $\tau_\hit$ we obtain (under no other assumptions than smoothness) is always better than the dependency on $n\tau_\mix$ of previous works (using gossip communications or a random walker), since $\tau_\hit\leq 2n\tau_\mix$ always holds.
As illustrated in Table~\ref{tb:times} on some known graphs, this inequality is rather loose when the connectivity decreases (\emph{i.e.}~the mixing time increases), so that the speedup our results lead to is even more effective on ill-connected graphs; the difference between the two can scale up to a factor $n$.
In fact, we prove in Appendix~\ref{app:basic} that for $d$-regular and symetric graphs, we have:
\begin{equation*}
    \tau_\hit\leq \frac{2|\cE|{\rm Diam}(G)}{d}\,,
\end{equation*}
where ${\rm Diam}$ is the diameter of $G$. The dependency $n\tau_\mix^2$ obtained in \cite{walkman} (the only work that does not make bounded gradient assumptions) is prohibitive when graph connectivity decreases ($n^3$ on the grid, $n^5$ on the cycle).
Our analysis does not rely on a reversibility assumption of the Markov chain, so that non symetric random walks can be used, therefore accelerating mixing; on the cycle for a non-symmetric random walk for instance, the hitting time decreases to $\cO(n)$.

\subsection{Using all gradient along the trajectory of $(v_t)$ is provably more efficient }

\citet{tao2018mcsgd} empirically motivated through empirical evidence the use of all gradients $\nabla f_{v_t}$ sampled along the trajectory of the Markov chain rather than waiting for the chain to mix before every stochastic gradient step in order to mimic the behavior of vanilla SGD.
However, their rates (as well as those of \cite{johansson2010randomizedincremental,duchi2012ergodic} and ours for MC-SGD) are functions of $S=T/\tau_\mix$, and of order $1/S+1/\sqrt{S}$.
These are exactly what one would obtain, by waiting for $\tau_\mix$ steps of the chain in order to have an approximate uniform sampling before each update!
Consequently, there are no theoretical ground or evidence for using all the gradients along the trajectory of the Markov chain with these results, other than by doing so, one does not do worse than waiting for the chain to mix to mimic vanilla SGD. This is exactly the approach taken by \citet{hendrikx2022tokens}: a gradient step is performed every $\tau_\mix$ random walk steps.
This is where MC-SAG and its guarantees that depend on $T/\tau_\hit$ come in place. 
Under our assumptions, the rate of SAG for finding approximate stationary points when waiting for the chain to mix before using a stochastic gradient is of order $n/S=n\tau_\mix/T$ where $S=T/\tau_\mix$ is the number of stochastic gradients used.
We obtain $\tau_\hit/T$ instead: hence, in cases where $\tau_\hit=o(n\tau_\mix)$, using all stochastic gradients along the trajectory of the Markov chain - instead of waiting for mixing before performing a stochastic gradient step - provably helps.
Hence, we here provided a realistic scenario where using all stochastic gradients proves to accelerate the rate; this was previously noticed in another setting with RER-SGD (SGD with reverse-experience replay, \cite{kowshik2021streaming}).

\subsection{Running-time complexity and robustness to ‘‘stragglers''}

The total time it takes to run random walk-based decentralized algorithms depends on $T_\vtow$, the time it takes to compute a gradient at $v$, and then the communication time to send it to $w$.
Using ergodicity of the Markov chain, the time ${\rm time}^{\rm MC}(T)$ it takes to run MC-SAG or MC-SGD for $T$ iterations verifies:
\begin{equation*}
    \frac{{\rm time}^{\rm MC}(T)}{T}\to\sum_{(v,w)\in\cV^2}\pi_v P_{v,w}T_\vtow\,,
\end{equation*}
where the limit is a weighted sum of the local computation/communication times, with weights summing to 1. Random-walk based decentralized algorithms are therefore robust to slow edges or nodes (‘‘stragglers''), a property that synchronous gossip algorithms do not verify (their time complexity depends on $\max_{v,w}T_{\vtow})$, while studying asynchronous gossip is notoriously difficult \citep{even2021delays}.

\paragraph{Numerical illustration of our theory}

We present in Appendix~\ref{sec:exp} two experiments on synthetic problems, comparing MC-SAG and MC-SGD to gossip-based and token baselines. We consider two settings (a well-connected graph with homogeneous functions,  an ill-connected graph with heterogeneous functions) in an effort to illustrate how these two difficulties (graph connectivity and data-heterogeneity) are both bypassed by the MC-SAG algorithm.


\paragraph{Conclusion}
Without variance reduction and under bounded data-heterogeneity assumptions, SGD under MC sampling is slowed down by a factor $\tau_\mix$, due to increased sampling variance. Using variance-reduction techniques, we obtain faster rates, that depend on $\tau_\hit$ rather than $n\tau_\mix$, which one would have expected by directly extending known results in the \emph{i.i.d.}~setting to our MC sampling schemes. Leveraging such a dependency yields a fast token algorithm (MC-SAG), robust to both ill-connectivity of the graph and data-heterogeneity.

\paragraph{Aknowledgments}
I deeply thank Hadrien Hendrikx for many interesting discussions on the subject and for all his help in writing this paper, as well as Anastasiia Koloskova and Edwige Cyffers for helpful discussions.

\bibliographystyle{plainnat} 
\bibliography{refs}

\newpage

\appendix

\onecolumn

\section{Preliminary results}\label{app:basic}

\subsection{Mixing time and relaxation time, mixing time and hitting time}

We first begin by the two following lemmas. The first one is very classical, and bounds the mixing time in terms of $1/\lambda_P$, in the case where the chain is reversible; we provide a proof for completeness. Note that if the chain is reversible, \emph{we still have a linear decay} \cite{LevinPeresWilmer2006}. 
The second lemma we provide bounds the hitting time of the Markov chain with the mixing time.
This result is somewhat less classical, and is not present in the classical Markov chain literature surveys.

\begin{lemma}[$\tau_\mix$ and $\lambda_P$]
    For any $\eps>0$, if the chain is reversible:
    \begin{equation*}
        \tau_\mix(\eps)\leq\left\lceil \frac{1}{\lambda_P}\ln(\eps^{-1}\pi_{\min}^{-1})\right\rceil\,,
    \end{equation*}
    so that $\tau_\mix\leq \left\lceil \frac{1}{\lambda_P}\ln(\pi_{\min}^{-2}/2)\right\rceil$.
\end{lemma}
\begin{proof}
    We have:
    \begin{align*}
        \dd_{\rm TV}(P^t\pi_0,\pi)&=\frac{1}{2}\sum_{w\in\cV}|(P^t\pi_0)_{w}-\pi_w|\\
        &\leq\frac{1}{2\pi_{\min}}\sum_{w\in\cV}\pi_w|(P^t\pi_0)_{w}-\pi_w|\\
        &\leq \frac{1}{2\pi_{\min}}\sqrt{\NRM{P^t\pi_0 -\pi}_\pi^2}\\
        &\leq \frac{(1-\lambda_P)^{t}}{2\pi_{\min}}\NRM{\pi_0-\pi}_\pi\\
        &\leq \frac{(1-\lambda_P)^{t}}{\pi_{\min}}\,,
    \end{align*}
    so that $|(P^t)_{v,w}-\pi_w|\leq \eps$ for $t\geq \lambda_P^{-1}\ln(\pi_{\min}^{-1}\eps^{-1}/2)$.
\end{proof}

\begin{lemma}[Mixing times and hitting times]
    \begin{equation*}
        \tau_\hit\leq 2 \pi_{\min}^{-1}\tau_\mix\,,
    \end{equation*}
    so that if $\pi$ is the uniform distribution over $\cV$, $\tau_\hit\leq 2n\tau_\mix$.
\end{lemma}
\begin{proof}
    for any  $v,w\in\cV$,
    \begin{align*}
        \esp{\tau_w|v_0=v}&=\sum_{k\geq 1}\proba{\tau_w\geq k|v_0=v}\\
        &\leq \sum_{\ell\geq 0}\proba{\tau_w > \ell\tau_\mix|v_0=v}\,.
    \end{align*}
    Then, for $\ell\geq 0$, 
    \begin{align*}
        \proba{\tau_w > (\ell+1)\tau_\mix|v_0=v}=\proba{\tau_w > (\ell+1)\tau_\mix|\tau_w>\ell\tau_\mix, v_0=v}\proba{\tau_w>\ell\tau_\mix| v_0=v}\,,
    \end{align*}
    and, conditioning on $v_{\ell\tau_\mix}$, $\proba{\tau_w > (\ell+1)\tau_\mix|\tau_w>\ell\tau_\mix, v_{\ell\tau_\mix}}\leq \proba{v_{(\ell+1)\tau\mix}\ne w|v_{\ell\tau_\mix}}$. 
    By definition of $\tau_\mix$, we have that $\proba{v_{(\ell+1)\tau\mix}\ne w|v_{\ell\tau_\mix}}\leq (1-\pi_w/2)$, so that:
    \begin{equation*}
        \proba{\tau_w > (\ell+1)\tau_\mix|v_0=v}\leq (1-\pi_w/2)\proba{\tau_w > \ell\tau_\mix|v_0=v} \,,
    \end{equation*}
    and $\proba{\tau_w > \ell\tau_\mix|v_0=v}\leq (1-\pi_w/2)^\ell$ by recursion. Finally,
    \begin{align*}
        \esp{\tau_w|v_0=v}&\leq \tau_\mix\sum_{\ell\geq 0}(1-\pi_w/2)^\ell\\
        &\leq \frac{2\tau_\mix}{\pi_w}\,,
    \end{align*}
    concluding the proof by taking the maximum over $w$.
\end{proof}

\subsection{Matthews' bound for cover times}

The following result bounds the cover time of the Markov chain: it is in fact closely related to its hitting time, and the two differ with a most a factor $\ln(n)$.
This surprising result is proved in a very elegant way in the survey \citet{LevinPeresWilmer2006}, using the famous Matthews' method \cite{matthews-covering}.

\begin{theorem}[Matthews' bound for cover times]
    The hitting and cover times of the Markov chain verify:
    \begin{equation*}
        \tau_\cov\leq \left(\sum_{k=1}^{n-1} \frac{1}{k}\right)\tau_\hit\,.
    \end{equation*}
\end{theorem}

\subsection{A bound on the hitting time of regular and symetric graphs}

Using results from \citet{hittings}, we relate the hitting time of symmetric regular graphs (in a sense that we define below) to well-known graph-related quantities: number of edges $|\cE|$, diameter $\delta$ and degree $d$.

\begin{lemma}[Bounding hitting times of regular graphs] Let $(v_t)$ be the simple random walk on a $d$-regular graph $G$ of diameter $\delta$, that satisfies the following symetry property: for any $\edgeuv,\edgevw\in\cE$, there exists a graph automorphism that maps $v$ to $w$.
Then, we have:
\begin{equation*}
    \tau_\hit\leq \frac{2|\cE|\delta}{d}
\end{equation*}
\end{lemma}
\begin{proof}
    Using Theorem~2.1 of \citet{hittings}, for $\edgevw\in\cE$, we have 
    \begin{equation*}
        \esp{\tau_w|v_0=v}= \frac{2|\cE|}{d}\,,
    \end{equation*}
    where $|\cE|$ is the number of edges in the graph.
    Let $v$ and $w$ in $\cV$, at distance $\delta'\leq\delta$.
    There exists nodes $v=v(0),v(1)\ldots,v(\delta'-1),v(\delta')=w$ such that for all $0\leq s<\delta'$, $\set{v(s),v(s+1)}\in\cE$, and by using the Markov property:
    \begin{equation*}
        \esp{\tau_w|v_0=v}\leq\sum_{s<\delta'}\esp{\tau_{v(s+1)}v_0=v(s)}\leq \delta \frac{2|\cE|}{d}\,.
    \end{equation*}
\end{proof}

\subsection{Two miscellaneous lemmas}

We finally end this ‘‘preliminary results'' section with the two following lemmas, that we help us conclude the proof of Theorem~\ref{thm:mcSAG_varying}. 
The first lemma will lead to a bound on $\Big(\sum_{t<T}\gamma_t\Big)^{-1}$ where $\gamma_t$ is the adaptive stepsize policy defined in Equation~\eqref{eq:stepsizes}, while the second one is used to conclude the proof of Theorem~\ref{thm:mcSAG_varying} to show that a remaining term is non-positive.

\begin{lemma}\label{lem:stepsizes_bound}
    For $t\geq 0$ and $v\in\cV$, let $p_v(t)=\inf\set{s>t|v_s=v}$  and $d_v(t)=\sup\set{s<t|v_s=v}$ be the next and the last previous iterates for which $v_t=v$ ($d_v(t)=0$ by convention, if $v$ has not yet been visited).
    Assume that $(v_t)$ has stationary distribution $\pi$.
    For $t\geq0$, let $A_t=\sup_{v\in\cV} \big(t-d_v(t)\big)$ and $B_t=\sup_{v\in\cV} \big(p_v(t)-t\big)$. We have:
    \begin{equation*}
        \esp{B_t|v_t=v}\leq \tau_\cov\,,\quad \forall v\in\cV\,,
    \end{equation*}
    and for $T\geq1$:
    \begin{equation*}
        \sum_{t<T}\esp{A_t}\leq T\tau_\cov\,.
    \end{equation*}
\end{lemma}
\begin{proof}
    The first bound on $B_t$ is obtained using the Markov property of the chain, and by definition of $\tau_\cov$.
    We have:
    \begin{equation*}
        \esp{A_t}=\sum_{k=0}^{t-1}\proba{A_t\geq k}\,.
    \end{equation*}
    For $t\geq0$ fixed, we denote $d=\inf_vd_v(t)$, so that $A_t=t-d$. We then have the equality between the following events:
    \begin{equation*}
        \set{A_t\geq k}=\set{t-d\geq k}=\set{d\leq t-k}=\set{B_{t-k}\geq k}\,,
    \end{equation*}
    that all coincide with the event ‘‘there exists some $v\in\cV$ such that for all $t-k\leq s\leq t$, $v_s\ne v$''.
    Summing over $t<T$:
    \begin{align*}
        \sum_{t<T}\esp{A_t}&=\sum_{t<T}\sum_{k<t}\proba{B_{t-k}\geq k}\\
        &=\sum_{\ell<T}\sum_{s<T-\ell} \proba{B_\ell \geq s}\\
        &\leq \sum_{\ell<T}\sum_{s\geq 0} \proba{B_\ell \geq s}\\
        &\leq \sum_{\ell<T} \esp{B_\ell}\\
        &\leq T\tau_\cov\,.
    \end{align*}
\end{proof}

\begin{lemma}\label{lem:frac_rdm}
	Let $(a_t)_{t\geq0},(b_t)_{t\geq0}$ be two sequences of real-valued random variables. 
    Let $(\cF_t)_{t\geq 0}$ be a filtration. 
    Assume that $b_t$ is positive and $\cF_t$-measurable for all $t$, and that $\esp{a_t|\cF_t}\leq 0$.
	Then, denoting $H_T=\frac{\sum_{t=0}^T a_t}{\sum_{t=0}^Tb_t}$, the sequence $(\esp{H_T})_{T\geq0}$ is non-increasing, so that $\esp{H_T}\leq 0$ for all $T$.
\end{lemma}
\begin{proof}
    For fixed $T\geq1$, we have, using the fact tha $b_t$ is $\cF_T$ measurable for $t\leq T$
    \begin{align*}
        \esp{H_T|\cF_T}&= \frac{ \esp{a_T|\cF_T}+ \sum_{t<T}\esp{a_t|\cF_T}}{\sum_{t\leq T}b_t}\\
        &\leq \frac{\sum_{t<T}\esp{a_t|\cF_T}}{\sum_{t\leq T}b_t}\\
        &\leq \frac{\sum_{t<T}\esp{a_t|\cF_T}}{\sum_{t< T}b_t}\,,
    \end{align*}
    using $\esp{a_T|\cF_T}\leq 0$ and $b_T>0$. Consequently, taking the mean, we obtain $\esp{H_T}\leq \esp{H_{T-1}}$.
\end{proof}

\section{Lower bound}\label{app:lower}

We prove the smooth non-convex version of Theorem~\ref{thm:lower}; the convex cases are proved in a similar way using exactly the same arguments, and the ‘‘most difficult function in the world'', as defined by \citet{nesterov_book}, rather than the one used by \citet{carmon2021lower}, albeit the two are closely related.

\begin{proof}[Proof of Theorem~\ref{thm:lower}]
    For $\xx\in\ell_2$ and $k\in\N$, denote by $\xx(k)$ its $k^{\rm th}$ coordinate.
    We split the function defined in Section 3.2 of \citet{carmon2021lower} (inspired by the ‘‘most difficult function in the world'' of \citet{nesterov_book}) between two nodes $v,w\in\cV$ maximizing $\esp{\tau_w|v_0=v}$, by setting $\pi_vf_v(\xx)=\frac{1}{2}\sum_{k\geq 1} 2\xx(2k)^2-2\xx(2k-1)\xx(2k) + \frac{1}{2}\alpha \xx(0)^2 -b \xx(0) + \frac{\alpha}{2}$ and $\pi_wf_w(\xx)=\frac{1}{2}\sum_{k\geq0} 2\xx(2k+1)^2 - 2\xx(2k+1)\xx(2k)$ for some $b,\alpha>0$.
    Then, we define $T_0=\tau_v$ and for~$t\geq0$,
    \begin{align*}
        &T_{2k+1}=\inf\set{t\geq T_{2k}\,|\, v_t=w}\quad\text{and}\\
        & T_{2k+2}=\inf\set{t\geq T_{2k+1}\,|\, v_t=w}\,.
    \end{align*}
    The second step of the proof is somewhat classical, and consists in observing that the black-box constraints of the algorithm together with the construction of the functions $f_v$ and $f_w$ defined in the proof sketch of Section~\ref{sec:lower} imply that:
    \begin{equation*}
        \left\{\begin{aligned}
            & \text{if}\quad v_t=v \quad \text{and}\quad \left\{\begin{aligned} 
                &\cM_t\supset{\rm Span}(\ee_i,i\leq 2k-1)\,\quad \text{then}\quad \cM_{t+1}\supset{\rm Span}(\ee_i,i\leq 2k)\,,\\
                &\cM_t\subset{\rm Span}(\ee_i,i\leq 2k)\,\quad \text{then}\quad \cM_{t+1}\subset{\rm Span}(\ee_i,i\leq 2k)\,,
            \end{aligned}\right. \\
            & \text{if}\quad v_t=w \quad \text{and} \quad 
            \left\{\begin{aligned} 
                &\cM_t\supset{\rm Span}(\ee_i,i\leq 2k)\,\quad \text{then}\quad \cM_{t+1}\supset{\rm Span}(\ee_i,i\leq 2k+1)\,,\\
                &\cM_t\subset{\rm Span}(\ee_i,i\leq 2k+1)\,\quad \text{then}\quad \cM_{t+1}\subset{\rm Span}(\ee_i,i\leq 2k+1)\,,
            \end{aligned}\right. \\
            & \text{if}\quad v_t\notin\edgevw, \quad \text{then} \quad \cM_t=\cM_{t+1}\,.
        \end{aligned}\right.
    \end{equation*}
    In other words, even dimensions are discovered by node $v$, while odd ones are discovered by node~$w$. The dimension $\R\ee_0$ is discovered by node $v$ thanks to the term $-b\ee_0$.
    Using Theorem~1 of \citet{carmon2021lower}, for a right choice of parameters $\alpha,b>0$, $f$ is $L$-smooth and satisfies $f(\xx_0)-\inf_\xx f(\xx)\leq \Delta$, together with, any $k$ and any $\xx\in\cM_t\subset{\rm Span}(\ee_i,i\leq 2k)$, 
    \begin{equation*}
        \NRM{\nabla f(\xx)}^2=\frac{L\Delta}{16k^2}\,.
    \end{equation*}
    This lower bound proof technique is explained in a detailed and enlightening fashion in Chapter 3.5 of  \citet{bubeck2014convex}.

    Then, the final and more technical step of the proof consists in upper bounding $\E k(t)$. If $(T_{k+1}-T_k)_{k\geq0}$ were independent from $k(t)$, using $\esp{T_{k+1}-T_k}=\tau_\hit$ for $k$ even, we would directly obtain $t\geq \esp{T_{k(t)}}\geq \esp{k(t)}-1)\tau_\hit/2$. 
    However, these random variables are not independent: since tail effects can happen, we need a finite second moment for hitting times, and the proof is a bit trickier.
    First, note that:
    \begin{equation*}
        \esp{k(t)}=\sum_{0\leq k\leq t} \proba{k(t)\geq k}= \sum_{0\leq k\leq t}\proba{T_k\leq t}\,. 
    \end{equation*}
    Let $(X_\ell)_{\ell\geq0}$ be \emph{i.i.d.} random variables of same law as $\tau_w$ conditioned on $v_0=v$.
    We have $\esp{X_\ell}=\esp{\tau_w|v_0=v}=\tau_\hit$, and $\var(X_\ell)<\infty$ (by assumption).
    Let $S_k=\sum_{\ell=0}^{k-1}X_\ell$ ($S_k$ has the same law as $\sum_{\ell=0}^{k-1}T_{2k+1}-T_{2k}$ ), so that, using the Markov property of $(v_t)$, $T_k$ stochastically dominates $S_{\lfloor k/2\rfloor}$. 
    Hence, $\proba{T_k\leq t}\leq \proba{S_{\lfloor k/2\rfloor}\leq t}$. Then, using Chebychev inequality, for any $\ell\geq0$ and for $t$ such that $\ell\tau_\hit\geq t$, we have:
    \begin{align*}
        \proba{S_{\ell}\leq t}&=\proba{S_\ell -\ell\tau_\hit \leq t-\ell\tau_\hit}\\
        &=\proba{(S_\ell -\ell\tau_\hit)^2 \leq (t-\ell\tau_\hit)^2}\\
        &\leq \frac{\ell\var(X_0)}{(t-\ell\tau_\hit)^2}\,.
    \end{align*}
    We then have:
    \begin{align*}
        \esp{k(t)}&\leq 2 \sum_{0\leq \ell \leq t/2} \proba{S_\ell\leq t}\\
        &= 2 \sum_{0\leq \ell \leq 2t/\tau_\hit}\proba{S_\ell\leq t} + 2 \sum_{2t/\tau_\hit\leq \ell \leq t/2}\proba{S_\ell\leq t}\\
        &\leq \frac{4t}{\tau_\hit} + 2 \sum_{2t/\tau_\hit\leq \ell \leq t/2}\frac{\ell\var(X_0)}{(t-\ell\tau_\hit)^2}\,.
    \end{align*}
    We finally show that the second term stays bounded:
    \begin{align*}
        \sum_{2t/\tau_\hit\leq \ell \leq t/2}\frac{\ell}{(t-\ell\tau_\hit)^2}&=\frac{1}{\tau_\hit^2}\sum_{0\leq \ell \leq t/2-2t/\tau_\hit}\frac{\ell+2t/\tau_\hit}{(\ell + t/\tau_\hit)^2}\\
        &=\frac{1}{\tau_\hit^2}\sum_{0\leq \ell \leq t/2-2t/\tau_\hit}\frac{\ell}{(\ell + t/\tau_\hit)^2} + \frac{2}{\tau_\hit^2}\sum_{0\leq \ell \leq t/2-2t/\tau_\hit}\frac{t/\tau_\hit}{(\ell + t/\tau_\hit)^2}\,.
    \end{align*}
    First, using a comparison with a continuous sum, we have:
    \begin{equation*}
        \sum_{0\leq \ell \leq t}\frac{\ell}{(\ell + t/\tau_\hit)^2}\leq \sum_{0\leq \ell \leq t}\frac{1}{(\ell + t/\tau_\hit)}\leq \ln(\frac{t}{t/\tau_\hit})=\ln(\tau_\hit)\,,
    \end{equation*}
    since for $a,x>0$, $\int_0^{ax} \frac{y\dd y}{(y+a)^2}\leq \int_0^{ax} \frac{\dd y}{(y+a)}=\ln(x)$. Finally, using $\sum_{\ell\geq 1}\frac{1}{(a+\ell)^2}\leq \int_0^\infty\frac{\dd y}{(y+a)^2}=\frac{1}{a}$, we bound the second sum as:
    \begin{equation*}
        \sum_{0\leq \ell \leq t/2-2t/\tau_\hit}\frac{t/\tau_\hit}{(\ell + t/\tau_\hit)^2}\leq \frac{\tau_\hit}{t}+1\,.
    \end{equation*}
    Wrapping our arguments together, we end up with:
    \begin{equation*}
        \esp{k(t)}\leq \frac{4t}{\tau_\hit} + \frac{2\var(\tau_v)}{\tau_\hit^2}(\ln(\tau_\hit) + 1 + \frac{\tau_\hit}{t})\,.
    \end{equation*}
    For $t$ big enough, we end up with $\esp{k(t)}\leq 5t/\tau_\hit$, so that since $\E\NRM{\nabla f(\xx_t)}^2\geq L\Delta/(16\esp{k(t)}^2)$ as explained in the main text, we have:
    \begin{equation*}
        \NRM{\nabla f(\xx_t)}^2=\Omega\left(\frac{L\Delta\tau_\hit^2}{t^2}\right)\,.
    \end{equation*}
\end{proof}

\section{Markov chain stochastic gradient descent: proof of Theorem~\ref{thm:mcsgd}}\label{app:mcsgd}


\textbf{The following proofs in this Appendix section are valid for finite as well as infinite state spaces $\cV$.
}

We start by proving the following bound on $\esp{\NRM{\nabla f_{v_t}(\xx_t)}^2}$.
Note that this bound can be used for any $t\geq\tau_\mix$. 
\begin{lemma}
    For $t\geq0$ and if $v_t\sim\pi_t$ for $\dd_{\rm TV}(\pi_t,\pi)\leq \pi_{\min}/2$, we have:
    \begin{equation*}
        \esp{\NRM{\nabla f_{v_t}(\xx_t)}^2}\leq 3\bar\sigma^2 + 2\esp{\NRM{\nabla f(\xx_t)}^2}\,.
    \end{equation*}
\end{lemma}
\begin{proof}[Proof of the Lemma]
    We have for any $v\in\cV$ that $\proba{v_t=v}\leq \pi_v +\pi_v/2=3\pi_v/2$, so that
    \begin{align*}
        \esp{\NRM{\nabla f_{v_t}(\xx_t)}^2}&\leq 2\esp{\NRM{\nabla f_{v_t}(\xx_t)-\nabla f(\xx_t)}^2}+2\esp{\NRM{\nabla f(\xx_t)}^2}\\
        &\leq 2\sum_{v\in\cV}\proba{v_t=v}\sigma_v^2+2\esp{\NRM{\nabla f(\xx_t)}^2}\\
        &=3\bar\sigma^2+2\esp{\NRM{\nabla f(\xx_t)}^2}\,.
    \end{align*}
\end{proof}
The proof borrows ideas from both the analyses of delayed SGD \citep{mania2016perturbed} and SGD with biased gradients \citep{even2022sample}, thus refining MC-SGD initial analysis \citep{johansson2010randomizedincremental}. 
While a biased gradient analysis would not yield convergence to an $\eps$-stationary point for arbitrary $\eps$ (at every iterations, biases are non-negligible and can be arbitrary high), by enforcing a delay $\tau$ (of order $\tau_\mix$) in the analysis, we manage to take advantage of the ergodicity of the biases.

\subsection{Smooth non-convex case of Theorem~\ref{thm:mcsgd}}

\begin{proof}[Proof of Theorem~\ref{thm:mcsgd}.\ref{thm:mcsgd_nonsmooth}]
	Denoting $F_t=\E f(\xx_t)-f(\xx^\star)$, we have using smoothness:
	\begin{equation*}
		F_{t+1}-F_t\leq -\gamma\esp{\langle \nabla f_{v_t}(\xx_t),\nabla f(\xx_t)\rangle} +\frac{\gamma^2L}{2}\esp{\NRM{\nabla f_{v_t}(\xx_t)}^2}\,.
	\end{equation*}
	For the first term on the righthandside of the inequality, assuming that $t\geq\tau$ for some $\tau>0$ we explicit later in the proof:
	\begin{align*}
		\esp{-\gamma\langle \nabla f_{v_t}(\xx_t),\nabla f(\xx_t)\rangle}&=\esp{-\gamma\langle \nabla f_{v_t}(\xx_{t-\tau}),\nabla f(\xx_{t-\tau})\rangle} + \esp{-\gamma\langle \nabla f_{v_t}(\xx_t),\nabla f(\xx_t)-\nabla f(\xx_{t-\tau})\rangle}\\
        &\quad + \esp{-\gamma\langle \nabla f_{v_t}(\xx_t)-\nabla f_{v_t}(\xx_{t-\tau}),\nabla f(\xx_{t-\tau})\rangle}\,.
	\end{align*}
	First, we condition the first term on the filtration up to time $t-\tau$:
	\begin{align*}
		\esp{-\gamma\langle \nabla f_{v_t}(\xx_{t-\tau}),\nabla f(\xx_{t-\tau})\rangle}&=\esp{-\gamma\langle \E_{t-\tau}\nabla f_{v_t}(\xx_{t-\tau}),\nabla f(\xx_{t-\tau})\rangle}\\
		&\leq -\frac{\gamma}{2}\esp{\NRM{\E_{t-\tau}\nabla f_{v_{t-\tau}}(\xx_t)}^2} + \frac{\gamma}{2}\esp{\NRM{\nabla f(\xx_{t-\tau})-\E_{t-\tau}\nabla f_{v_t}(\xx_{t-\tau})}}\\
		&\quad-\frac{\gamma}{2}\esp{\NRM{\nabla f(\xx_{t-\tau})}^2}\,.
	\end{align*}
	Then, for $\tau\geq\tau_{\rm mix}(\pi_{\min}\eps)$, using the following lemma, we have, for $\eps<1/2$:
    \begin{equation*}
        \esp{-\gamma\langle \E_{t-\tau}\nabla f_{v_t}(\xx_{t-\tau}),\nabla f(\xx_{t-\tau})\rangle}\leq -\frac{\gamma}{4}\esp{\NRM{\nabla f(\xx_{t-\tau})}^2} + \gamma\eps^2\bar\sigma^2\,.
        \end{equation*}
    \begin{lemma}
        For $\tau\geq\tau_\mix(\eps\pi_{\min})$ and $t\geq\tau$,
        \begin{equation*}
            \esp{\NRM{\E_{t-\tau}\nabla f_{v_t}(\xx_{t-\tau})-\nabla f(\xx_{t-\tau})}^2}\leq 2\eps^2\esp{\NRM{\nabla f(\xx_{t-\tau})}^2}+2\eps^2 \bar\sigma^2\,.
        \end{equation*}
    \end{lemma}
    \begin{proof}[Proof of the Lemma] We have:
        \begin{align*}
            \esp{\NRM{\E_{t-\tau}\nabla f_{v_t}(\xx_{t-\tau})-\nabla f(\xx_{t-\tau})}^2}&=\esp{\NRM{\sum_{v\in\cV} (\proba{v_t=v|\xx_{t-\tau}}-\pi_v) \nabla f_v(\xx_{t-\tau})}^2}\\
            &\leq \eps^2\sum_{v\in\cV}\pi_v\esp{\NRM{ \nabla f_v(\xx_{t-\tau})}^2}\,,
        \end{align*}
        where we used $|\proba{v_t=v|\xx_{t-\tau}}-\pi_v|\leq \eps\pi_v$ and convexity of the squared Euclidean norm. For that last term, 
    \begin{align*}
        \sum_{v\in\cV}\pi_v\esp{\NRM{ \nabla f_v(\xx_{t-\tau})}^2}&\leq \sum_{v\in\cV}2\pi_v\left(\esp{\NRM{ \nabla f(\xx_{t-\tau})}^2} + \sigma_v^2\right)\\
        &=2\esp{\NRM{\nabla f(\xx_{t-\tau})}^2} +2\bar\sigma^2\,,
    \end{align*}
    concluding the proof of the Lemma.
    \end{proof}
	Using gradient Lipschitzness and writing $\xx_t-\xx_{t-\tau}=-\gamma\sum_{s=\max(t-\tau,0)}^{t-1}\nabla f_{v_s}(\xx_s)$, we have:
	\begin{align*}
		\esp{-\gamma\langle \nabla f_{v_t}(\xx_t),\nabla f(\xx_t)-\nabla f(\xx_{t-\tau})\rangle}&\leq \gamma^2L\esp{\NRM{\nabla f_{v_t}(\xx_t)}\NRM{\sum_{s=\max(t-\tau,0)}^{t-1}\nabla f_{v_s}(\xx_s)}}\\
		&\leq \frac{\gamma^2L}{2} (\tau\esp{\NRM{\nabla f_{v_t}(\xx_t)}^2}+\sum_{s=\max(t-\tau,0)}^{t-1}\esp{\NRM{\nabla f_{v_s}(\xx_s)}^2}\big)\,.
	\end{align*}
	Similarly,
    \begin{equation*}
        \esp{-\gamma\langle \nabla f_{v_t}(\xx_t)-\nabla f_{v_t}(\xx_{t-\tau}),\nabla f(\xx_{t-\tau})\rangle}\leq\frac{\gamma^2L}{2} (\tau\esp{\NRM{\nabla f(\xx_{t-\tau})}^2}+\sum_{s=\max(t-\tau,0)}^{t-1}\esp{\NRM{\nabla f_{v_s}(\xx_s)}^2}\big)\,.
    \end{equation*}
	Wrapping things up, we obtain, for $t\geq\tau$ and $\tau\geq\tau_\mix$:
    \begin{align*}
		F_{t+1}-F_t&\leq -\frac{\gamma}{4}\esp{\NRM{\nabla f(\xx_{t-\tau})}^2}+ \gamma\eps^2\bar\sigma^2\\
        &\quad + \frac{\gamma^2 L}{2}\big((\tau+1)\esp{\NRM{\nabla f_{v_t}(\xx_t)}^2} +\tau \esp{\NRM{\nabla f(\xx_{t-\tau})}^2}+2\sum_{s=\max(t-\tau,0)}^{t-1}\esp{\NRM{\nabla f_{v_s}(\xx_s)}^2} \big)\\
        &\leq -\frac{\gamma}{4}\esp{\NRM{\nabla f(\xx_{t-\tau})}^2}+ \gamma\eps^2\bar\sigma^2 + (3\tau+1)\frac{\gamma^2L}{2}\\
        &\quad + \frac{\gamma^2 L}{2}\big((\tau+1)\esp{\NRM{\nabla f(\xx_t)}^2} +\tau \esp{\NRM{\nabla f(\xx_{t-\tau})}^2}+2\sum_{s=\max(t-\tau,0)}^{t-1}\esp{\NRM{\nabla f(\xx_s)}^2} \big)\,.
        %
	\end{align*}
	Summing for $\tau\leq t<T$: 
	\begin{equation*}
		\frac{1}{T}\sum_{ \tau\leq t<T}\esp{\NRM{\nabla f(\xx_{t-\tau})}^2}\leq \frac{4F_{\tau}}{\gamma T}+  \frac{1}{T}\sum_{\tau\leq t<T}6\gamma L\tau\big(\esp{\NRM{\nabla f(\xx_t)}^2}+\esp{\NRM{\nabla f(\xx_{t-\tau})}^2}\big) +2\big(2\eps^2+\gamma(3\tau+1)\big)\bar\sigma^2\,,
	\end{equation*}
	leading to, for $\gamma \leq \frac{1}{12 L\tau}$:
	\begin{align}\label{eq:proof-mcsgd}
		\frac{1}{T}\sum_{ t<T-\tau}\esp{\NRM{\nabla f(\xx_t)}^2}&\leq\frac{4F_{\tau}}{\gamma T}+ \frac{6\gamma L\tau}{T}\sum_{T-\tau\leq t <T}\esp{\NRM{\nabla f(\xx_t)}^2}+ 2\big(2\eps^2+\gamma(3\tau+1)\big)\bar\sigma^2\,.
	\end{align}
    We now prove that for any $t\geq 0$, we have $\sup_{t\leq s\leq t+\tau}\esp{\NRM{\nabla f(\xx_s)}^2} \leq 4\esp{\NRM{\nabla f(\xx_t)}^2} + 8\gamma^2L^2\tau^2\sigma^2$. Let $t\leq s<t+\tau$.
    \begin{align*}
        \esp{\NRM{\nabla f(\xx_s)}^2} & \leq 2\esp{\NRM{\nabla f(\xx_t)}^2} + 2 \esp{\NRM{\nabla f(\xx_s)-\nabla f(\xx_t)}^2}\\
        &\leq 2\esp{\NRM{\nabla f(\xx_t)}^2} + 2L^2\gamma^2\esp{\sum_{r=t}^{s-1}\NRM{\nabla f_{v_r}(\xx_r)}^2}\\
        &\leq 2\esp{\NRM{\nabla f(\xx_t)}^2} + 4L^2\gamma^2\tau\sum_{r=t}^{s-1}\esp{\NRM{\nabla f(\xx_r)}^2} + \bar\sigma^2\\
        &\leq 2\esp{\NRM{\nabla f(\xx_t)}^2} + 4L^2\gamma^2\tau^2 (\sup_{t\leq s\leq t+\tau}\esp{\NRM{\nabla f(\xx_s)}^2} + \bar\sigma^2)\,,
    \end{align*}
    leading to the desired result for $\gamma\leq 1/(8L\tau)$. Plugging this in \eqref{eq:proof-mcsgd}:
    \begin{align*}
        \frac{1}{T}\sum_{ t<T-\tau}\esp{\NRM{\nabla f(\xx_t)}^2}&\leq\frac{4F_{\tau}}{\gamma T}+ \frac{24\gamma L\tau}{T}\sum_{T-\tau\leq t <T}\esp{\NRM{\nabla f(\xx_{t-\tau})}^2}\\
        &\quad +\frac{\tau}{T}4L^2\gamma^2\tau^2\bar\sigma^2 + 2\big(2\eps^2+\gamma(3\tau+1)\big)L\bar\sigma^2\,.
    \end{align*}
    Now, for $\gamma=\min(1/(48L\tau), \sqrt{F_0/(TL\tau\bar\sigma^2)} )$, $\eps=1/\sqrt{T}$, and so $\tau=\tau_\mix\ln(T)$, we have:
    \begin{equation}\label{eq:proof_mcsgd_conclu_nonconvex}
        \frac{1}{T}\sum_{ t<T-\tau}\esp{\NRM{\nabla f(\xx_t)}^2}\leq \frac{196\tau LF_\tau}{T} + 7\sqrt{\frac{LF_0\bar\sigma^2}{T}}\,.
    \end{equation}
    We now upper bound $F_\tau$. For any $t$ and $\gamma<1/(2L)$,
    \begin{equation*}
        F_{t+1}-F_t\leq \frac{\gamma}{2}\esp{\NRM{\nabla f(\xx_t)-\nabla f_{v_t}(\xx_t)}^2}\leq  \gamma\sigma_{\max}^2/2\,,
    \end{equation*}
    where the first inequality is a simplified version of the descent lemma with biased gradient at the beggining of this proof, and the second inequality uses the initialization properties of $v_0$.
    Thus, we obtain $F_\tau\leq F_0 +\gamma\tau\sigma_{\max}^2/2\leq F_0+\sigma_{\max}^2/L$ for our choice of $\gamma$.
    We thus conclude by plugging this in \eqref{eq:proof_mcsgd_conclu_nonconvex} applied for $T+\tau$ instead of $T$, yielding the desired result.

    The condition for the upper-bound we proved above to be true, namely $T\geq\tau=\tau_\mix\ln(T)$, is always satisfied for $T\geq 2\tau_\mix\ln(\tau_\mix)$. Indeed, if $T\leq \tau_\mix^2$, then $\tau_\mix\ln(T)\leq 2\tau_\mix\ln(\tau_\mix)\leq T$, and otherwise we have $\tau_\mix\ln(T)\leq \sqrt{T}\ln(T)\leq T$. This concludes the proof, and
    $\tilde\xx_0$ in the Theorem corresponds to $\xx_\tau$.
\end{proof}

\subsection{Under a $\mu$-PL inequality}

\begin{proof}[Proof of Theorem~\ref{thm:mcsgd_PL}.\ref{thm:mcsgd_PL}] 
    We start from:
    \begin{align*}
        F_{t+1}-F_t&\leq -\frac{\gamma}{4}\esp{\NRM{\nabla f(\xx_{t-\tau})}^2}+ \gamma\eps^2\bar\sigma^2 + (3\tau+1)\frac{\gamma^2L}{2}\\
        &\quad + \frac{\gamma^2 L}{2}\big((\tau+1)\esp{\NRM{\nabla f(\xx_t)}^2} +\tau \esp{\NRM{\nabla f(\xx_{t-\tau})}^2}+2\sum_{s=\max(t-\tau,0)}^{t-1}\esp{\NRM{\nabla f(\xx_s)}^2} \big)\,.
    \end{align*}
    If $f$ satisfies a $\mu$-PL inequality, then $-\esp{\NRM{\nabla f(\xx_{t-\tau})}^2}\leq -2\mu F_{t-\tau}$, so that, for some $\alpha\in(0,1)$:
    \begin{align*}
        F_{t+1}-F_t&\leq -\frac{\alpha\gamma\mu}{4}F_{t-\tau} - \frac{(1-\alpha)\gamma}{8}\esp{\NRM{\nabla f(\xx_{t-\tau})}^2}+ \gamma\eps^2\bar\sigma^2 + (3\tau+1)\frac{\gamma^2L}{2}\bar\sigma^2\\
        &\quad + \frac{\gamma^2 L}{2}\big((\tau+1)\esp{\NRM{\nabla f(\xx_t)}^2} +\tau \esp{\NRM{\nabla f(\xx_{t-\tau})}^2}+2\sum_{s=\max(t-\tau,0)}^{t-1}\esp{\NRM{\nabla f(\xx_s)}^2} \big)\,.
    \end{align*}
    For $t\geq0$, let $P_t=(1-\alpha\gamma\mu/4)^{-t}$. 
    We multiply the above expression by $P_{t+1}$ and sum for $t<T$, hoping for cancellations. 
    For $T\geq \tau$:
    \begin{align*}
        \sum_{\tau\leq t<T}P_{t+1}\big(F_t-F_{t+1} -\frac{\alpha\gamma\mu}{4}F_{t-\tau} \big)& = \sum_{\tau\leq t<T}P_{t+1}\big((1-\frac{\alpha\gamma\mu}{4})F_t-F_{t+1} + \frac{\alpha\gamma\mu}{4}(F_t-F_{t-\tau}) \big)\\
        & = \sum_{\tau\leq t<T}P_{t}F_t - \sum_{\tau+1\leq t\leq T}P_{t}F_t\\
        &\quad + \frac{\alpha\gamma\mu}{4}\sum_{\tau\leq t<T}P_{t+1} F_t - \frac{\alpha\gamma\mu}{4}\sum_{\tau\leq t<T}P_{t+1} F_{t-\tau} \\
        &\leq P_\tau F_\tau-P_TF_T + \frac{\alpha\gamma\mu}{4}\sum_{\tau\leq t<T}P_{t+1} F_t - \frac{P_\tau \alpha\gamma\mu}{4}\sum_{0\leq t<T-\tau}P_{t+1} F_{t}\\
        &\leq P_\tau F_\tau-P_TF_T + \frac{\alpha\gamma\mu}{4}\sum_{T-\tau\leq t<T}P_{t+1} F_t\\
        &\leq P_\tau F_\tau-P_TF_T + \frac{\alpha\gamma}{8}\sum_{T-\tau\leq t<T}P_{t+1} \esp{\NRM{\nabla f(\xx_t)}^2}\,,
    \end{align*}
    using the $\mu$-PL inequality. For $t\geq0$, we denote $R_t= \esp{\NRM{\nabla f(\xx_t)}^2}$. 
    We now handle the ‘‘$R_t$'' terms.
    \begin{align*}
        - \sum_{\tau\leq t<T}\frac{(1-\alpha)\gamma}{8}P_{t+1}R_{t-\tau}& + \sum_{\tau\leq t<T}\frac{\gamma^2 L}{2}\big((\tau+1)P_{t+1}R_t+\tau P_{t+1} R_{t-\tau}+2\sum_{s=t-\tau}^{t-1}P_{t+1}R_s \big)\\
        &\leq - \sum_{0\leq t<T-\tau}\frac{(1-\alpha)\gamma}{8}P_{t+\tau+1}R_{t}\\
        & + \frac{\gamma^2L}{2}\left( (\tau+1)\sum_{\tau\leq t<T}P_{t+1}R_t + \tau \sum_{0\leq t<T-\tau}P_{t+1}R_t +2\tau\sum_{ t<T}R_t P_{t+\tau}      \right)\\
        &= - \sum_{0\leq t<T-\tau}P_{t+1}R_{t}\gamma\Big(\frac{(1-\alpha)}{8}P_\tau -  \frac{\gamma L}{2}(2\tau+1 + 2\tau P_{\tau-1})\Big)\\
        &+ \frac{\gamma^2 L}{2} \sum_{T-\tau\leq t<T} \big((\tau+1+2\tau P_{\tau-1}) \big)P_{t+1}R_t\\
        &\leq - \sum_{0\leq t<T-\tau}\frac{(1-\alpha)\gamma}{16}P_{t+\tau+1}R_{t}\\
        &+ \quad\frac{(1-\alpha)\gamma}{16\beta} \sum_{T-\tau\leq t<T}P_{t+1}R_t\,,
    \end{align*}
    if $\gamma$ satisfies $\gamma\leq \frac{1-\alpha}{8\beta L(5\tau+1)}$ and $P_\tau\leq 2$, for some $\beta\geq1$. 
    Since for $\gamma\mu\leq1$, $P_\tau\leq e^{\tau\mu\gamma}$, $P_\tau\leq2$ can be ensured with $\gamma\leq \frac{1}{2\tau L}$.
    All in one, we have:
    \begin{align*}
        0&\leq P_\tau F_\tau-P_TF_T + \gamma\Big(\frac{\alpha}{8}+\frac{1-\alpha}{16\beta}\Big)\sum_{T-\tau\leq t<T}P_{t+1}R_t\\
        &\quad - \sum_{0\leq t<T-\tau}\frac{(1-\alpha)\gamma}{16}P_{t+\tau+1}R_{t}\\
        &\quad + \Big(\gamma\eps^2\bar\sigma^2 + (3\tau+1)\frac{\gamma^2L}{2}\bar\sigma^2 \Big) \sum_{\tau\leq t<T} P_{t+1}\,.
    \end{align*}
    Using what we proved in the previous proof, we have $R_t\leq 4R_{t-\tau} + 8\gamma^2L^2\tau^2\sigma^2$ for $T-\tau\leq t<T$, so that:
    \begin{align*}
        \gamma\Big(\frac{\alpha}{8}+\frac{1-\alpha}{16\beta}\Big)\sum_{T-\tau\leq t<T}P_{t+1}R_t&\leq 4\gamma\Big(\frac{\alpha}{8}+\frac{1-\alpha}{16\beta}\Big)\sum_{T-2\tau\leq t<T-\tau}P_{t+\tau+1}R_t \\
        &\quad + 8\gamma^2L^2\tau^2\sigma^2\gamma\Big(\frac{\alpha}{8}+\frac{1-\alpha}{16\beta}\Big)\sum_{T-2\tau\leq t<T-\tau}P_{t+\tau+1}\,.
    \end{align*}
    Consequently, for $4\gamma\Big(\frac{\alpha}{8}+\frac{1-\alpha}{16\beta}\Big)\leq \frac{1-\alpha}{16}\gamma$, which can be ensured with $\alpha=1/16$ and $\beta=8$, we have:
    \begin{align*}
        0&\leq P_\tau F_\tau-P_TF_T + \frac{1}{8}\gamma^3L^2\tau^2\sigma^2\sum_{T-2\tau\leq t<T-\tau}P_{t+\tau+1} \\
        &\quad + \Big(\gamma\eps^2 + (3\tau+1)\frac{\gamma^2L}{2} \Big)\bar\sigma^2 \sum_{\tau\leq t<T} P_{t+1}\,,
    \end{align*}
    so that:
    \begin{align*}
        F_T&\leq F_\tau /P_{T-\tau} + \gamma^2\bar\sigma^2L\Big(\frac{\eps^2}{L\gamma} +  \frac{3\tau+1 }{2} + \frac{\gamma L\tau^2}{8}  \Big) \frac{\sum_{t\leq T}P_t}{P_T}\\
        &\leq 2 F_\tau /P_T + \frac{2\gamma \bar\sigma^2}{\mu}L\Big(\frac{\eps^2}{L\gamma} +  \frac{3\tau+1 }{2} + \frac{\gamma L\tau^2}{8}  \Big)\\
        &\leq 2 F_\tau /P_T + \frac{2\gamma \bar\sigma^2}{\mu}L(\frac{\eps^2}{L\gamma}+\frac{1}{2} + 2\tau)\,.
    \end{align*}
    Finally, using $F_\tau\leq F_0 +\sigma_{\max}^2/L$, if $\gamma\leq \frac{1}{64(5\tau+1)}$ where $\tau=\tau_\mix(\eps)$:
    \begin{equation*}
        F_T\leq 2 (F_0+\sigma_{\max}^2/L) (1-\frac{\gamma\mu}{8})^T + \frac{2\gamma \bar\sigma^2}{\mu}(\frac{\eps^2}{L\gamma}+\frac{1}{2} + 2\tau)
    \end{equation*}
    We thus choose $\eps=\sqrt{1/T}$ so that $\tau\leq\tau_\mix\ln(T)$, and stepsize $\gamma=\min(\frac{8\ln(T(F_0+\bar\sigma^2)/\bar\sigma^2)}{\mu T},\frac{1}{64(5\tau+1)})$  , leading to the desired result for $c=64\times 6=384$.

    The same discussion than in the smooth non-convex proof regarding the condition $T\geq\tau$ applies here.
\end{proof}

\section{Markov chain SGD: results in the interpolation regime}\label{app:mcsgd_interpolation}

\subsection{Proof of Theorem~\ref{thm:mcsgd_interpolation}}

We begin by proving the following lemma.

\begin{lemma}\label{lem:noise-optim}
    For any $T\geq1$, we have:
    \begin{equation*}
        \esp{\NRM{\sum_{t<T}\nabla f_{v_t}(\xx^\star)}^2}\leq T \sigma_\star^2 + \sigma_{\star}^2\sum_{t<T}\dd_{\rm TV}(P_{v_0,\cdot}^t,\pi^\star) + 2\sigma_\star^2\sum_{s<t<T}\dd_{\rm TV}(t-s)\,,
    \end{equation*}
    where $\dd_{\rm TV}(r)=\sup\set{\dd_{\rm TV}((P^r)_{v,\cdot},\pi^\star)\,,v\in\cV}$ for $r\in\N$, so that:
    \begin{equation*}
        \esp{\NRM{\sum_{t<T}\nabla f_{v_t}(\xx^\star)}^2}\leq \sigma_\star^2\big(4\tau_\mix(1/4) + T(1+8\tau_\mix(1/4))\big)\,.
    \end{equation*}
\end{lemma}
\begin{proof}
    We have:
    \begin{align*}
        \esp{\NRM{\sum_{t<T}\nabla f_{v_t}(\xx^\star)}^2}&=\esp{\left(\sum_{t<T}\nabla f_{v_t}(\xx^\star)\right)^\top\left(\sum_{t<T}\nabla f_{v_t}(\xx^\star)\right)}\\
        &= \sum_{t<T}\esp{\NRM{\nabla f_{v_t}(\xx^\star)}^2} + 2\sum_{s<t<T}\esp{\langle \nabla f_{v_s}(\xx^\star),\nabla f_{v_t}(\xx^\star)\rangle}\,.
    \end{align*}
    Denote $\GG_\star=(\nabla f_v(\xx^\star))_{v\in\cV}\in\R^{\cV\times d}$. For the first term above, 
    \begin{align*}
        \sum_{t<T}\esp{\NRM{\nabla f_{v_t}(\xx^\star)}^2}&=\sum_{t<T}\esp{\NRM{\GG_{\star,v_s}}^2}\\
        &=\sum_{t<T}\left(\sigma_\star^2+\sum_{v\in\cV}\big(\proba{v_t=v}-\pi_v\big)\NRM{\GG_{\star,v}}^2\right)\\
        &\leq T \sigma_\star^2 + \sigma_{\star}^2\sum_{t<T}\dd_{\rm TV}(P_{v_0,\cdot}^t,\pi^\star)\,.
    \end{align*}
    Finally,
    \begin{align*}
        \sum_{s<t<T}\esp{\langle \nabla f_{v_s}(\xx^\star),\nabla f_{v_t}(\xx^\star)\rangle}&=\sum_{s<t<T}\esp{\langle \GG_{\star,v_s},\GG_{\star,v_t}\rangle}\\
        &=\sum_{s<t<T}\sum_{v,w\in\cV} (P^s)_{v_0,v}(P^{t-s})_{v,w} \GG_{\star,v}^\top \GG_{\star,w}\\
        &=\sum_{s<t<T}\sum_{v,w\in\cV} (P^s)_{v_0,v}\big((P^{t-s})_{v,w}-\frac{1}{n}\big) \GG_{\star,v}^\top \GG_{\star,w}\\
        &\leq\sum_{s<t<T}\sum_{v,w\in\cV} (P^s)_{v_0,v}\big|(P^{t-s})_{v,w}-\frac{1}{n}\big|\sigma_\star^2\\
        &=\sigma_\star^2\sum_{s<t<T}\sum_{v\in\cV} (P^s)_{v_0,v}\sum_{w\in\cV}\big|(P^{t-s})_{v,w}-\frac{1}{n}\big|\\
        &\leq\sigma_\star^2\sum_{s<t<T}\dd_{\rm TV}(t-s)\,,
    \end{align*}
    where $\dd_{\rm TV}(r)=\sup\set{\dd_{\rm TV}((P^r)_{v,\cdot},\pi^\star)\,,v\in\cV}$ for $r\in\N$.

    We finally bound $\sum_{t<T}\dd_{\rm TV}(t)$. Using Chapter 4.5 of \citet{LevinPeresWilmer2006}, we have $\tau_\mix(\eps)\leq (\log_2(\eps^{-1})+1)\tau_\mix(1/4)$, so that for any $t\geq 0$, $d_\TV(t)\leq 2^{-t/\tau_\mix(1/4) + 1}$. Hence,
    \begin{equation*}
        \sum_{t<T}\dd_{\rm TV}(t)\leq  \frac{2}{1-2^{-1/\tau_\mix(1/4)}}\leq 4\tau_\mix(1/4)\,,
    \end{equation*}
    and 
    \begin{equation*}
        \sum_{s<t<T}\dd_{\rm TV}(t-s)\leq 4T\tau_\mix(1/4)\,,
    \end{equation*}
    concluding the proof.
\end{proof}

\begin{lemma}
    For any $\yy_t\in\R^d$ and $t\geq0$, denoting $\yy_{t+1}=\yy_t-\gamma\nabla f_{v_t}(\xx^\star)$, we have:
    \begin{equation*}
        \NRM{\xx_{t+1}-\yy_{t+1}}^2\leq (1-\gamma\mu)\NRM{\xx_t-\yy_t}^2 + \gamma L \NRM{\yy_t-\xx^\star}^2\,.
    \end{equation*}
\end{lemma}
\begin{proof}
    Denote $f_t=f_{v_t}(\cdot)$.
    We expand:
    \begin{align*}
        \NRM{\xx_{t+1}-\yy_{t+1}}^2&=\NRM{\xx_{t}-\yy_{t}}^2 - 2\gamma\langle \nabla f_t(\xx_t)-\nabla f_t(\xx^\star),\xx_t-\yy_t\rangle + \gamma^2\NRM{\nabla f_t(\xx_t)-\nabla f_t(\xx^\star)}^2\,.
    \end{align*}
    Using the \emph{three points equality} as \citet{mischenko2020randomreshuffling}, we have:
    \begin{equation*}
        - 2\gamma\langle \nabla f_t(\xx_t)-\nabla f_t(\xx^\star),\xx_t-\yy_t\rangle=-2\gamma D_{f_t}(\yy_t,\xx_t)-2\gamma D_{f_t}(\xx_t,\xx^\star)+2\gamma D_{f_t}(\yy_t,\xx^\star)\,.
    \end{equation*}
    First, $-2\gamma D_{f_t}(\yy_t,\xx_t)\leq -\gamma\mu\NRM{\xx_t-\yy_t}^2$ using strong convexity.
    Then, $-2\gamma D_{f_t}(\xx_t,\xx^\star)$ cancels the term $\gamma^2\NRM{\nabla f_t(\xx_t)-\nabla f_t(\xx^\star)}^2\leq 2\gamma^2LD_{f_t}(\xx_t,\xx^\star)$ for $\gamma\leq 1/L$, using smoothness of $f_t$.
    And finally, using smoothness again, $2\gamma D_{f_t}(\yy_t,\xx^\star)\leq \gamma L\NRM{\yy_t-\xx^\star}^2$, concluding the proof.
\end{proof}

\begin{proof}[Proof of Theorem~\ref{thm:mcsgd_interpolation}]
    Fix some $y_0\in\R^d$ and let $(\yy_t)$ be defined with the recursion
    \begin{equation*}
        \yy_{t+1}=\yy_t-\gamma\nabla f_{v_t}(\xx^\star,\xi_t)\,.
    \end{equation*} 
    Unrolling the previous Lemma, we have, for a fixed time horizon $T$:
    \begin{equation*}
        \NRM{\xx_T-\yy_T}^2\leq (1-\gamma\mu)^T\NRM{\xx_0-\yy_0}^2+\gamma L \sum_{t<T}(1-\gamma\mu)^{T-t}\NRM{\yy_t-\xx^\star}^2\,.
    \end{equation*}
    This is possible, since the descent lemma is deterministic, in the sense that no expectations are taken so far.
    Since we want control over the distance to the optimum, we wish to have $\yy_T=\xx^\star$, leading to:
    \begin{equation*}
        \yy_0=\xx^\star+\gamma\sum_{t<T}\nabla f_{v_t}(\xx^\star,\xi_t)\,,\quad \yy_s=\xx^\star+\gamma\sum_{s\leq t<T}\nabla f_{v_t}(\xx^\star,\xi_t)\,,\quad s<T\,.
    \end{equation*}
    We thus have:
    \begin{align*}
        \E\NRM{\xx_T-\xx^\star}^2&\leq 2(1-\gamma\mu)^T\left(\E\NRM{\xx_0-\xx^\star}^2+ \gamma^2\esp{\NRM{\sum_{s<T}\nabla f_{v_s}(\xx^\star)}^2}\right)\\
        &\quad+\gamma^3 L \sum_{t<T}(1-\gamma\mu)^{T-t}\esp{\NRM{\sum_{t\leq s<T}\nabla f_{v_s}(\xx^\star)}^2}\,.
    \end{align*}
    Using Lemma~\ref{lem:noise-optim}, we have for any $t\leq T$:
    \begin{equation*}
        \esp{\NRM{\sum_{t\leq s<T}\nabla f_{v_s}(\xx^\star)}^2}\leq CT\tau_\mix\sigma_\star^2\,,
    \end{equation*}
    for $C=13$.
    Hence,
    \begin{align*}
        \E\NRM{\xx_T-\xx^\star}^2&\leq 2(1-\gamma\mu)^T\left(\E\NRM{\xx_0-\xx^\star}^2+ \gamma^2TC\tau_\mix\sigma_\star^2\right)\\
        &\quad+\gamma^3 L \sum_{t<T}(1-\gamma\mu)^{T-t}(T-t)C\tau_\mix\sigma_\star^2\\
        &\leq 2(1-\gamma\mu)^T\NRM{\xx_0-\xx^\star}^2+\frac{3\gamma L}{\mu^2}C\tau_\mix\sigma_\star^2\,,
    \end{align*}
    using $\sum_{t<T}(1-x)^tt\leq 1/x^2$ and $(1-z)^xx\leq \frac{1}{eu}\leq 1/(2u)$.
    Finally, for a stepsize choice of 
    \begin{equation*}
        \gamma=\min\left(\frac{1}{L}\,,\, \frac{1}{T\mu} \ln\left(T\frac{\NRM{\xx_0-\xx^\star}^2}{\frac{3\gamma L}{\mu^2}C\tau_\mix\sigma_\star^2}\right)\right)\,,
    \end{equation*}
    we obtain:
    \begin{equation*}
        \esp{\NRM{\xx_T-\xx^\star}}\leq 2e^{-\frac{\mu T}{L}}\NRM{\xx_0-\xx^\star}^2 + \cO\left(\frac{L\tau_\mix\sigma_\star^2}{\mu^3T}\right)\,.
    \end{equation*}
\end{proof}

\subsection{Proof of Proposition~\ref{prop:lower_mcsgd}}

\begin{proof}
Consider the graph $G$ on the set of nodes $\cV=\set{0,1}$, with probability transitions $p_{01}=p_{10}=p$ and $p_{00}=p_{11}=1-p$, for some small $p\in(0,1)$.
The relaxation time $\tau_\mix(1/4)$ of this graph scales as $1/p$.

Consider now $f_0(\xx)=\frac{1}{2}(\xx-1)^2$ and $f_1(\xx)=\frac{1}{2}(\xx+1)^2$ for $\xx\in\R$, so that $\xx^\star=0$.
For $(v_t)$ a Markov chain with the given transition probabilities, started at $v_0$ following the uniform (stationary) distribution on $\cV$, let $\xx_t$ be generated with MC-SGD: $\xx_{t+1}=\xx_t-\gamma \nabla f_{v_t}(\xx_t)$ and $\xx_0=1$, \emph{i.e.},
\begin{equation*}
    \xx_T= (1-\gamma)^T -\gamma \sum_{t<T}(1-\gamma)^{T-t-1}\zeta_t\,,
\end{equation*}
where $\zeta_t\in\set{-1,1}$ takes value $1$ if $v_t=0$ and value $-1$ if $v_t=0$.
We have:
\begin{equation*}
    \esp{(\xx_T-\xx^\star)^2}= (1-\gamma)^{2T}+\gamma^2\sum_{s<t<T}(1-\gamma)^{2T-t-s-2}\esp{\zeta_s\zeta_t}\,.
\end{equation*}
We compute this second term, and show that it is non-negative for $p\leq 1/2$ and of order $\frac{\gamma}{4p}$, so that to reach a given precision $\eps>0$, is required $\gamma\leq 4p\eps$ and thus to make the first term small, $T$ must verify $T=\Omega(1/(2p\eps))$, concluding our reasonning.

For $s<t$, we have $\esp{\zeta_s\zeta_t}=2\proba{v_{t-s}=v_0|v_0\sim\pi^\star}-1$/ Denoting $z_k=\proba{v_{k}=v_0|v_0\sim\pi^\star}$, we have $z_{k+1}=pz_k+(1-p)(1-z_k)$ and $z_0=1$, so that $z_k=\frac{1}{2}(1+(1-2p)^k)$ for $k\geq0$.
This leads to:
\begin{align*}
    \gamma^2\sum_{s<t<T}(1-\gamma)^{2T-t-s-2}\esp{\zeta_s\zeta_t}&=\gamma^2\frac{(1-\gamma)(1-2p)}{1-(1-\gamma)(1-2p)}\\
    &\times \left( \frac{1-(1-\gamma)^{2T}}{1-(1-\gamma)^2}-(1-2p)\frac{(1-\gamma)^T-(1-2p)^T}{2p-\gamma}(1-\gamma)^T\right)
\end{align*}
For $\eps\to0$, in order to have $\esp{(\xx_T-\xx^\star)^2}\leq\eps$, is required $(1-\gamma)^T\leq \eps$ so that $\gamma T\to \infty$. Under $\gamma T\to \infty$, we have
\begin{equation*}
    \gamma^2\sum_{s<t<T}(1-\gamma)^{2T-t-s-2}\esp{\zeta_s\zeta_t}\sim \frac{\gamma}{4p}\,.
\end{equation*}
Finally, to reach precision $\eps$, this quantity needs to be upper-bounded by $\eps(1+o(1))$, so that $\gamma \leq 4p\eps^{-1}(1+o(1))$ is necessary.
Plugging this in $(1-\gamma)^{2T}\leq \eps$ yields $T=\Tilde\Omega(p\eps^{-1})$, the desired result.
\end{proof}

\section{With local noise: proof of Theorem~\ref{thm:local_noise}}

The proof follows the exact same steps as the proof of Theorem~\ref{thm:mcsgd_interpolation}.

\begin{proof}
    First, note that, using Lemma~15 in \citet*{stich2021errorfeedback}, we have:
    \begin{equation*}
        \esp{\NRM{\sum_{t<T} \nabla_\xx F_{v_t} (\xx^\star,\xi_t)  }^2} \leq 2 \esp{\NRM{\sum_{t<T} \nabla f_{v_t} (\xx^\star)  }^2} +2T\sigma_{\rm local}^2\,.
    \end{equation*}

    We then have the following lemma, proved exactly as in the previous section.
    \begin{lemma}
        For any $\yy_t\in\R^d$ and $t\geq0$, denoting $\yy_{t+1}=\yy_t-\gamma\nabla_\xx F_{v_t}(\xx^\star,\xi_t)$, we have:
        \begin{equation*}
            \NRM{\xx_{t+1}-\yy_{t+1}}^2\leq (1-\gamma\mu)\NRM{\xx_t-\yy_t}^2 + \gamma L \NRM{\yy_t-\xx^\star}^2\,.
        \end{equation*}
    \end{lemma}
    This leads to:
    \begin{equation*}
        \NRM{\xx_T-\yy_T}^2\leq (1-\gamma\mu)^T\NRM{\xx_0-\yy_0}^2+\gamma L \sum_{t<T}(1-\gamma\mu)^{T-t}\NRM{\yy_t-\xx^\star}^2\,,
    \end{equation*}
    for
    \begin{equation*}
        \yy_0=\xx^\star+\gamma\sum_{t<T}\nabla_\xx F_{v_t}(\xx^\star,\xi_t)\,,\quad \yy_s=\xx^\star+\gamma\sum_{s\leq t<T}\nabla_\xx F_{v_t}(\xx^\star,\xi_t)\,,\quad s<T\,.
    \end{equation*}
    We thus have:
    \begin{align*}
        \E\NRM{\xx_T-\xx^\star}^2&\leq 2(1-\gamma\mu)^T\left(\E\NRM{\xx_0-\xx^\star}^2+ \gamma^2\esp{\NRM{\sum_{s<T}\nabla_\xx F_{v_s}(\xx^\star,\xi_s)}^2}\right)\\
        &\quad+\gamma^3 L \sum_{t<T}(1-\gamma\mu)^{T-t}\esp{\NRM{\sum_{t\leq s<T}\nabla_\xx F_{v_s}(\xx^\star,\xi_s)}^2}\,.
    \end{align*}
    To conclude, we use the first inequality of this proof, and Lemma~\ref{lem:noise-optim}, and proceed as in the proof of Theorem~\ref{thm:mcsgd_interpolation} and Corollary~\ref{thm:mcsgd_interpolation_step}.
\end{proof}

\section{MC-SAG: proof of Theorem~\ref{thm:mcSAG_varying}}\label{app:mcSAG}

\subsection{With perfect initialization: Theorem~\ref{thm:mcSAG_varying}.\ref{thm:mcSAG_varying_init}}

We begin classically by proving a descent lemma. This lemma is deterministic, in the sense that not means $\E$ are present, and it therefore does not use the Markovian properties of the Markov chain.
MC-SAG uses biased gradients, even in the case where $v_t$ are \emph{i.i.d.}, since the algorithm SAG \citep{SAG} is inherently biased (making it unbiased leads to the SAGA iterations \citep{saga}).

Let $\GG_t=\bar\hh_{t+1}$ for $t\geq0$, so that $\xx_{t+1}=\xx_t-\gamma_t\GG_t$.
We recall that for $v\in\cV$ and $t\geq0$, $p_v(t)$ is the next time (strictly) the chain hits node $v$, while $d_v(t)$ is either the last time the chain was at the state $v$ (if that happened), or $0$ in $v$ has not yet been visited.

\begin{lemma}\label{lem:descent_sag}
	Assume that $f$ is $L$-smooth. Then, for any $t\geq0$, we have:
	\begin{align*}
		f(\xx_{t+1})-f(\xx_t)\leq -\frac{\gamma_t}{2}\NRM{\nabla f(\xx_t)}^2 -\frac{\gamma_t}{4}\NRM{\GG_t}^2 + \frac{\gamma_t L^2}{2n} \sum_{v\in\cV}\NRM{\sum_{s=d_{v}(t)}^{t-1}\gamma_s\GG_s}^2\,.
	\end{align*} 
\end{lemma}

\begin{proof}
    For $t\geq 0$ and $v\in\cV$, let $p_v(t)=\inf\set{s>t|v_s=v}$  and $d_v(t)=\sup\set{s\leq t|v_s=v}$ be the next and the last previous iterates for which $v_t=v$ ($d_v(t)=0$ by convention, if $v$ has not yet been visited).
	Denote $F_t=\esp{f(\xx_t)-f(\xx^\star)}$. We have, using smoothness:
	\begin{equation*}
		f(\xx_{t+1})\leq f(\xx_t) -\gamma_t\langle \nabla f(\xx_t), \GG_t\rangle + \frac{\gamma_t^2L}{2}\NRM{\GG_t}^2\\
	\end{equation*}
	Together with $\langle \nabla f(\xx_t), \GG_t\rangle= \frac{1}{2}(\NRM{\nabla f(\xx_t)}^2+\NRM{\GG_t}^2-\NRM{\nabla f(\xx_t)-\GG_t}^2)$, we obtain:
	\begin{align*}
		f(\xx_{t+1})&\leq  f(\xx_t) -\frac{\gamma_t}{2}\NRM{\nabla f(\xx_t)}^2+\NRM{\GG_t}^2-\NRM{\nabla f(\xx_t)-\GG_t}^2+ \frac{\gamma_t^2L}{2}\NRM{\GG_t}^2\\
		&\leq f(\xx_t)-\frac{\gamma_t}{2}\NRM{\nabla f(\xx_t)}^2-\frac{\gamma_t}{4}\NRM{\GG_t}^2+\frac{\gamma_t}{2}\NRM{\nabla f(\xx_t)-\GG_t}^2\,,
	\end{align*}
	as long as $\gamma_t\leq 1/(2L)$.
	We thus need to upperbound the bias $\NRM{\nabla f(\xx_t)-\GG_t}^2$. We have:
	\begin{align*}
		\NRM{\nabla f(\xx_t)-\GG_t}^2&= \NRM{\frac{1}{n} \sum_{v\in\cV} \nabla f_v(\xx_{d_v(t)})-\nabla f_v(\xx_t)}^2\\
		&\leq \frac{1}{n} \sum_{v\in\cV}\NRM{\nabla f_v(\xx_{d_v(t)})-\nabla f_v(\xx_t)}^2\,.
	\end{align*}
	Fix some $v$ in $\cV$. We have $\NRM{\nabla f_v(\xx_t)-\nabla f_v(\xx_{d_{v}(t)})}^2\leq L^2\NRM{\sum_{s=d_{v}(t)}^{t-1}\gamma_s\GG_s}^2$, leading to:
	\begin{align*}
		f(\xx_{t+1})-f(\xx_t)&\leq -\frac{\gamma_t}{2}\NRM{\nabla f(\xx_t)}^2 -\frac{\gamma_t}{4}\NRM{\GG_t}^2 + \frac{\gamma_t L^2}{2n} \sum_{v\in\cV}\NRM{\sum_{s=d_{v}(t)}^{t-1}\gamma_s\GG_s}^2\,.
	\end{align*} 
\end{proof}

We now proceed with the proof of Theorem~\ref{thm:mcSAG_varying}.\ref{thm:mcSAG_varying_init}.

\begin{proof}[Proof of Theorem~\ref{thm:mcSAG_varying}.\ref{thm:mcSAG_varying_init}]
	We begin with
	\begin{align*}
		f(\xx_{t+1})-f(\xx_t)&\leq -\frac{\gamma_t}{2}\NRM{\nabla f(\xx_t)}^2 -\frac{\gamma_t}{4}\NRM{\GG_t}^2 +\frac{\gamma_t L^2}{2n} \sum_{v\in\cV}\NRM{\sum_{s=d_{v}(t)}^{t-1}\gamma_s\GG_s}^2\,,
	\end{align*} 
	as a starting point. For $v\in\cV$, 
	\begin{align*}
		\gamma_t\NRM{\sum_{s=d_{v}(t)}^{t-1}L^2\gamma_s\GG_s}^2&\leq \sum_{s=d_{v}(t)}^{t-1}(t-d_v(t))L^2\gamma_t\gamma_s^2\NRM{\GG_s}^2\\
		&\leq \sum_{s=d_{v}(t)}^{t-1}L\gamma_s^2\NRM{\GG_s}^2\,,
	\end{align*}
	since $\gamma_t\leq 1/(L(t-d_v(t)))$. Summing for $t<T$, we obtain:
	\begin{align*}
		\sum_{t<T}\frac{\gamma_t}{2}\NRM{\nabla f(\xx_t)}^2\leq F_0 - \sum_{t<T}\frac{\gamma_t}{2}\NRM{\GG_t}^2 + \sum_{t<T} \frac{1}{2n}\sum_{v\in\cV}\sum_{s=d_{v}(t)}^{t-1}L\gamma_s^2\NRM{\GG_s}^2\,.
	\end{align*}
	Then,
	\begin{align*}
		\sum_{t<T}\sum_{s=d_{v}(t)}^{t-1}L\gamma_s^2\NRM{\GG_s}^2 = \sum_{s<T} \NRM{\GG_s}^2L\gamma_s^2 (p_v(s)-s)\,.
	\end{align*}
	For $\gamma_s\leq 1/(2L\tau_{\rm hit})$, we have $\esp{\NRM{\GG_s}^2\gamma_s^2(p_v(s)-s)}\leq \frac{1}{2}\esp{\NRM{\GG_s}^2\gamma_s}$, so that:
	\begin{align*}
		\sum_{t<T}\frac{\gamma_t}{2}\NRM{\nabla f(\xx_t)}^2\leq F_0 + \sum_{t<T}\left(-\frac{\gamma_t}{2}\NRM{\GG_t}^2 + K_t\right)  \,,
	\end{align*}
	where $K_t=\frac{1}{4n}\sum_{v\in\cV}\NRM{\GG_t}^2L\gamma_t^2 (p_v(t)-t)$ verifies $\esp{K_t|\cF_t}\leq \NRM{\GG_t}^2\gamma_t/4$ since $\gamma_t\leq 1/(2\tau_{\rm hit})$, where $\cF_t$ is the filtration up to time $t$:
    \begin{align*}
        \esp{K_t|\cF_t}&=\esp{\frac{1}{4n}\sum_{v\in\cV}\NRM{\GG_t}^2L\gamma_t^2 (p_v(t)-t)|\cF_t}\\
        &=\frac{1}{4n}\sum_{v\in\cV}\NRM{\GG_t}^2L\gamma_t^2 \esp{(p_v(t)-t)|\cF_t} \quad \big(\GG_t\,,\,\gamma_t\text{ are $\cF_t$-measurable}\big)\\
        &\leq \frac{1}{4n}\sum_{v\in\cV}\NRM{\GG_t}^2L\gamma_t^2 \tau_\hit \quad \big(\text{since }\esp{(p_v(t)-t)|\cF_t}=\esp{(p_v(t)-t)|v_t}\leq \tau_\hit\big)\\
        &\leq \frac{1}{8}\gamma_t \NRM{\GG_t}^2 \quad \big(\text{since } \gamma_t\leq 1/(2\tau_\hit)\big)\,.
    \end{align*}
	Finally,
	\begin{align*}
		\esp{\min_{t<T}\NRM{\nabla f(\xx_t)}^2}\leq \esp{\frac{2F_0}{\sum_{t<T}\gamma_t}} + \esp{\frac{\sum_{t<T}\left(-\frac{\gamma_t}{2}\NRM{\GG_t}^2 + K_t\right)}{\sum_{t<T}\gamma_t}}  \,.
	\end{align*}
	Using Lemma~\ref{lem:frac_rdm} and the above bound on $\esp{K_t|\cF_t}$, that last term is non-positive. Using Jensen inequality, we have:
	\begin{align*}
		\esp{\min_{t<T}\NRM{\nabla f(\xx_t)}^2}&\leq \frac{2F_0}{T^2}\sum_{t<T}\esp{\gamma_t^{-1}}\,.
	\end{align*}
	Since $\gamma_t^{-1}=2L(\tau_{\rm hit} + \sup_{v\in\cV} (t-d_v(t)))$, we have $\sum_{t<T}\esp{\gamma_t^{-1}}\leq 2LT(\tau_{\rm hit} + \tau_{\rm cov})$ using Lemma~\ref{lem:stepsizes_bound}, concluding the proof.
\end{proof}

\subsection{With arbitrary initialization: proof of Theorem~\ref{thm:mcSAG_varying}.\ref{thm:mcSAG_varying_arbitrary}}

We no longer assume that $\GG_0=\nabla f(\xx_0)$ or that $\hh_v=\nabla f_v(\xx_0)$: $\hh_v$ are arbitrary and $\GG_0=\frac{1}{n}\sum_{v\in\cV}\hh_v$.
In that case, we have the following descent lemma.

\begin{lemma}\label{lem:descent_sag_2}
	Assume that $f$ is $L$-smooth. 
    Let 
    \begin{equation*}
        \tilde\tau_\cov=\inf\set{t\geq 1\quad \text{such that}\quad \set{v_0,\ldots,v_{t-1}}=\cV}\,.
    \end{equation*}
    Then, for any $t\geq0$, we have:
	\begin{align*}
		f(\xx_{t+1})-f(\xx_t)&\leq -\frac{\gamma_t}{2}\NRM{\nabla f(\xx_t)}^2 -\frac{\gamma_t}{4}\NRM{\GG_t}^2 + \frac{\gamma_t L^2}{n} \sum_{v\in\cV}\NRM{\sum_{s=d_{v}(t)}^{t-1}\gamma_s\GG_s}^2 + \frac{\gamma_t\one_{t<\tilde\tau_\cov}}{n}\sum_{v\in\cV}\NRM{\nabla f_v(\xx_0)-\hh_v}^2\,.
	\end{align*} 
\end{lemma}

\begin{proof}
As in Lemma~\ref{lem:descent_sag}, we have
	\begin{align*}
		f(\xx_{t+1})\leq f(\xx_t)-\frac{\gamma_t}{2}\NRM{\nabla f(\xx_t)}^2-\frac{\gamma_t}{4}\NRM{\GG_t}^2+\frac{\gamma_t}{2}\NRM{\nabla f(\xx_t)-\GG_t}^2\,,
	\end{align*}
	as long as $\gamma_t\leq 1/(2L)$.
	We thus again need to upperbound the bias $\NRM{\nabla f(\xx_t)-\GG_t}^2$. 
    For $t\geq \tilde\tau_\cov$, we have $\GG_t= \frac{1}{n} \sum_{v\in\cV} \nabla f_v(\xx_{d_v(t)})$, and so:
	\begin{align*}
		\NRM{\nabla f(\xx_t)-\GG_t}^2\leq\frac{1}{n} \sum_{v\in\cV}\NRM{\nabla f_v(\xx_{d_v(t)})-\nabla f_v(\xx_t)}^2\,,
	\end{align*}
    hence the result of Lemma~\ref{lem:descent_sag} holds for $t\leq \tilde\tau_\cov$ (and so does that of the Lemma we are proving).
    Then, for $t\leq\tilde\tau_\cov$, $\nabla f(\xx_t)-\GG_t=\frac{1}{n}\sum_{v\in\cV}\nabla f_v(\xx_t)-\hh_v^t$, where $\hh_v^t=\hh_v$ if $v$ hasn't been visited yet, or $\hh_v^t=\nabla f_v(\xx_{d_v(t)})$ otherwise. Hence,
    \begin{align*}
		\NRM{\nabla f(\xx_t)-\GG_t}^2&\leq\frac{1}{n} \sum_{v\in\cV}\NRM{\hh_v^t-\nabla f_v(\xx_t)}^2\\
        &\leq \frac{2}{n} \sum_{v\in\cV}\NRM{\nabla f_v(\xx_{d_v(t)})-\nabla f_v(\xx_t)}^2 + \frac{2}{n} \sum_{v\in\cV}\NRM{\nabla f_v(\xx_0)-\hh_v}^2
	\end{align*}
    Since $\NRM{\nabla f_v(\xx_t)-\nabla f_v(\xx_{d_{v}(t)})}^2\leq L^2\NRM{\sum_{s=d_{v}(t)}^{t-1}\gamma_s\GG_s}^2$, we have:
	\begin{align*}
		f(\xx_{t+1})-f(\xx_t)&\leq -\frac{\gamma_t}{2}\NRM{\nabla f(\xx_t)}^2 -\frac{\gamma_t}{4}\NRM{\GG_t}^2 + \frac{\gamma_t L^2}{n} \sum_{v\in\cV}\NRM{\sum_{s=d_{v}(t)}^{t-1}\gamma_s\GG_s}^2 + \frac{\gamma_t}{n}\sum_{v\in\cV}\NRM{\nabla f_v(\xx_0)-\hh_v}^2\,,
	\end{align*} 
    leading to the desired result.
\end{proof}

\begin{proof}[Proof of Theorem~\ref{thm:mcSAG_varying}.\ref{thm:mcSAG_varying_arbitrary}]
    Starting from
    \begin{align*}
		f(\xx_{t+1})-f(\xx_t)&\leq -\frac{\gamma_t}{2}\NRM{\nabla f(\xx_t)}^2 -\frac{\gamma_t}{4}\NRM{\GG_t}^2 + \frac{\gamma_t L^2}{n} \sum_{v\in\cV}\NRM{\sum_{s=d_{v}(t)}^{t-1}\gamma_s\GG_s}^2 + \frac{\gamma_t\one_{t<\tilde\tau_\cov}}{n}\sum_{v\in\cV}\NRM{\nabla f_v(\xx_0)-\hh_v}^2\,,
	\end{align*}
    and mimicking the proof with initialization, we obtain that for $\gamma_t^{-1}=4L(\tau_{\rm hit} + \sup_{v\in\cV} (t-d_v(t)))$, we have
    \begin{equation*}
        \esp{\min_{t<T}\NRM{\nabla f(\xx_t)}^2}\leq \frac{2F_0}{T^2}\sum_{t<T}\esp{\gamma_t^{-1}} + \frac{2\esp{\tilde\tau_\cov}}{T}\frac{1}{n}\sum_{v\in\cV}\NRM{\nabla f_v(\xx_0)-\hh_v}^2\,,
    \end{equation*}
    and we conclude the proof by noticing that $\esp{\tilde\tau_\cov}\leq\tau_\cov$.
\end{proof}

\section{Numerical illustration of our theory}

\label{sec:exp}

\paragraph{Setting} We place ourselves in the decentralized optimization setting on a graph $G$ with local functions $f_v$, and build two toy problems.
We compare our algorithms MC-SGD and MC-SAGA with Walkman \citep{walkman} and decentralized SGD (D-SGD in Figure~\ref{fig:exp}) \citep{pmlr-v119-koloskova20a,lian2019decentralizedmomemtum} with both randomized gossip communications and fixed gossip matrix.
We consider the non-convex loss function $\ell(\xx,a,b)=(\sigma(\xx^\top a)-b)^2/2$ where $\sigma(t)=1/(1+\exp(-t))$ as in \citet{nonconvexloss}. 
For $v\in\cV$, we take $f_v(x)=\ell(\xx,a_v,b_v)$ for $a_v$ and $b_v$ random variables.
In Figure~\ref{fig:homogeneous}, we take a connected random geometric graph of $n=50$ nodes in $[0,1]^2$ with radius parameter $\rho=0.3$ (nodes are connected if their distance is less than $\rho$).
We consider homogeneous data: $a_v,b_v$ are taken \emph{i.i.d.}, with $a_v\sim\cN(0,1)$ and $b_v$ uniform in $[0,1]$.
In Figure~\ref{fig:heterogeneous}, we take the cycle of $n=50$ nodes and consider heterogeneous data: for two opposite nodes in the graph ($v=0$ and $v=25$ \emph{e.g.}), $a_v$ are $\cN(0,1)$ and $b_v$ uniform in $[0,1]$ (and the functions are re-normalized), while they are taken equal to $0$ in the rest of the graph.
In Figure~\ref{fig:homogeneous}, the graph is well connected and Assumption~\ref{hyp:dissimilarities} is verified for a small enough $\bar\sigma^2$, so that all algorithms perform comparatively well; MC-SGD because of the homogeneity, the gossip-based ones and Walkman thanks to the connectivity.
However, decreasing the connectivity and increasing data heterogeneity in a pathological way, we obtain Figure~\ref{fig:heterogeneous}: MC-SGD fails due to function-heterogeneity ($\bar{\sigma}^2$ too big), while the three others are slowed down by their communication inefficiency, illustrating how depending only on $\tau_\hit$, MC-SAGA outperforms the other algorithms, that rather depend on $n\tau_\mix$ or $n\tau_\mix^2$.

\begin{figure}[h]
    \subfigure[Geometric graph, homogeneous functions]{
        \includegraphics[width=0.45\linewidth]{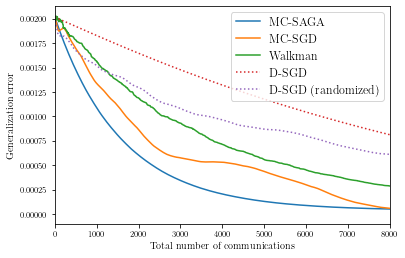}
        \label{fig:homogeneous}
    }
    \hfill
    \subfigure[Cycle graph, heterogeneous functions]{
        \includegraphics[width=0.45\linewidth]{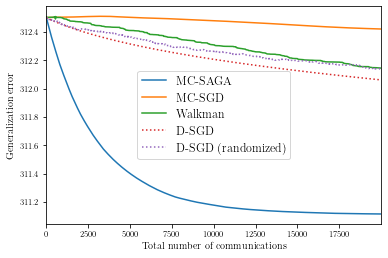}
        \label{fig:heterogeneous}
    }
    \caption{Comparison of MC-SAGA and MC-SGD with existing algorithms. In abscissa are the total number of communications, and in ordinate $f(\xx)-f(\xx^\star)$.\label{fig:exp}
    }
    \end{figure}

\end{document}

%% file: macros_old.tex
\usepackage[utf8]{inputenc} 
\usepackage[T1]{fontenc}    
\usepackage{url}            
\usepackage{booktabs}       
\usepackage{amsfonts}       
\usepackage{nicefrac}       
\usepackage{microtype}      
\usepackage{xcolor}         
\usepackage{graphicx}
\usepackage{subfigure}
\usepackage{dsfont}

\usepackage{todonotes}

\usepackage{amsmath,amsthm,amsfonts, amssymb}

\usepackage{threeparttable}

\usepackage{amsfonts}
\usepackage{amsthm}
\usepackage{amsmath}
\usepackage{amssymb}
\usepackage{mathrsfs}
\usepackage{amscd}
\usepackage{caption}
\usepackage{indentfirst}
\usepackage{cleveref}

\usepackage{graphicx}
\usepackage{subfigure}

\providecommand{\ee}{\mathbf{e}}

\renewcommand{\gg}{\mathbf{g}}
\providecommand{\GG}{\mathbf{G}}
\providecommand{\hh}{\mathbf{h}}

\providecommand{\xx}{\mathbf{x}}
\providecommand{\yy}{\mathbf{y}}

\newcommand{\E}{\mathbb{E}}

\newcommand{\cO}{\mathcal O}

\newcommand{\R}{\mathbb R}

\newcommand{\cE}{\mathcal{E}}
\renewcommand{\empty}{\varnothing}

\def\<#1,#2>{\langle #1,#2\rangle}

\usepackage{dsfont}

\renewcommand{\leq}{\leqslant}
\renewcommand{\geq}{\geqslant}

\def\<{\langle}
\def\>{\rangle}
\def\|{\Vert}

\def\eps{\varepsilon}
\def\var{{\rm var\,}}

\newcommand{\esp}[1]{\mathbb{E}\left[#1\right]}

\newcommand{\NRM}[1]{{{\left\| #1\right\|}}} 
\newcommand{\set}[1]{{{\left\{ #1\right\}}}} 
\newcommand{\proba}[1]{\mathbb{P}\left(#1\right)}

\newcommand{\cB}{\mathcal{B}}
\newcommand{\cN}{\mathcal{N}}

\newcommand{\cM}{\mathcal{M}}
\newcommand{\cV}{\mathcal{V}}

\newcommand{\dd}{{\rm d}}
\newcommand{\cF}{\mathcal{F}}

\newcommand{\N}{\mathbb{N}}
\newcommand{\cD}{\mathcal{D}}
\newcommand{\one}{\mathds{1}}

\usepackage{amsfonts}
\usepackage{amsthm}
\usepackage{amsmath}
\usepackage{amssymb}
\usepackage{mathrsfs}
\usepackage{amscd}
\usepackage{caption}
\usepackage{indentfirst}

\newcommand{\edgeuv}{{\{u,v\}}}
\newcommand{\edgevw}{{\{v,w\}}}

\newcommand{\vtow}{{v\to w}}

\newcommand{\mix}{{\rm mix}}
\newcommand{\hit}{{\rm hit}}
\newcommand{\cov}{{\rm cov}}
\newcommand{\TV}{{\rm TV}}

\newtheorem{theorem}{Theorem}
\newtheorem{assumption}{Assumption}
\newtheorem{definition}{Definition}
\newtheorem{lemma}{Lemma}
\newtheorem{proposition}{Proposition}
\newtheorem{remark}{Remark}
\newtheorem{corollary}{Corollary}